\documentclass{siamart251216}   

\usepackage{hyperref}
\usepackage{amsmath,amsfonts,amssymb,mathabx,bm}
\usepackage{graphicx,esint,ulem}
\usepackage{calc}
\usepackage{xcolor}
\usepackage{stmaryrd}   
\usepackage{wrapfig}
\usepackage{fancyvrb} 

\usepackage{color}

\usepackage{algorithm,algpseudocode}
\usepackage{algorithmicx}
\usepackage{hhline}     
\algrenewcommand\alglinenumber[1]{\footnotesize #1:} 

\usepackage{tcolorbox}
\usepackage[english]{babel}
\usepackage{tabu}
\usepackage{multirow}
\usepackage{mdframed}
\usepackage{cellspace} 
\usepackage{url}

\newcommand{\ul}{u_{\ell}}
\newcommand{\uh}{u_h}

\newcommand{\unh}{u_{nh}}
\newcommand{\ufh}{u_{fh}}

\newcommand{\Ap}{A^+}
\newcommand{\calNd}{\mathcal N_{\delta}}
\newcommand{\Kf}{K_f}
\newcommand{\Nf}{N_f}
\newcommand{\NL}{N_L}
\newcommand{\alphaF}{\alpha_F}

\newcommand{\sighat}{S}
\newcommand{\Nlam}{N_{\lambda}}
\newcommand{\sinc}{\text{sinc}}
\newcommand{\lambdam}{\lambda_{\max}}

\newcommand{\nm}{n_{\rm{max}}}

\definecolor{darkcyan}{rgb}{0.0, 0.55, 0.55}

\newcommand{\pt}{\partial_t}

\newcommand{\dt}{\Delta t}

\newcommand{\dk}{\Delta k}

\newcommand{\beq}{\begin{equation}}
\newcommand{\eeq}{\end{equation}}
\newcommand{\beqs}{\begin{equation*}}
\newcommand{\eeqs}{\end{equation*}}

\newcommand{\Nx}{N_x}

\newcommand{\Nt}{N_t}

\newcommand{\xv}{\bold x}

\newcommand{\yv}{\bold y}

\newcommand{\abs}[1]{\left\vert #1\right\vert}

\newcommand{\kv}{\bold k}

\newcommand{\calE}{\mathcal{E}}
\newcommand{\calH}{\mathcal{H}}
\newcommand{\norm}[1]{\Vert #1\Vert}

\newcommand{\calG}{\mathcal{G}}

\newcommand{\nv}{\bold n}
\newcommand{\mv}{\bold m}
\newcommand{\erf}{\text{erf}}

\newcommand{\ceil}[1]{\left\lceil#1 \right\rceil}
\newcommand{\Oh}[1]{\mathcal O\left(#1\right)}
\newcommand{\phiD}{\phi_{\Delta}}
\newcommand{\R}{\mathbb{R}}

\newsiamthm{assumption}{Assumption}       
\newsiamthm{example}{Example}             
\newsiamthm{lem}{Lemma}
\newsiamremark{remark}{Remark}               

\def\ccm{Center for Computational Mathematics, Flatiron Institute,
  New York, NY, 10010}
\def\nyu{Courant Institute of Mathematical Sciences,
  New York University, New York, NY, 10012}

\title{A spectral method for the rapid evaluation of hyperbolic potentials 
  in two dimensions using windowed Fourier projection}

\author{Nour G. Al Hassanieh%
  \textsuperscript{\dag}%
  \thanks{\nyu}
  \and
  Leslie Greengard%
    \textsuperscript{\dag}\footnotemark[1]
    \and
    Alex H. Barnett%
  \thanks{\ccm\
      (\email{nalhassanieh@flatironinstitute.org},
    \email{abarnett@flatironinstitute.org},
    \email{lgreengard@flatironinstitute.org}).}    
}
\begin{document}
\maketitle  

\begin{abstract}
  We present a fast algorithm for evaluating the (non-smooth) solution of the free-space
  two-dimensional (2D)
  scalar wave equation
  with many point sources, each with a high-frequency band-limited time signature.
  Such an algorithm is key to an efficient time-domain scattering
  solver using spatially-discretized hyperbolic layer potentials.
Given $M$ sources/targets and $N_t$ time steps, direct evaluation costs $O(M^2N_t^2)$,
due to the history dependence.
We develop a quasi-linear scaling algorithm
that splits the solution at a given time into (a) a non-smooth time-local part,
(b) a (smooth) {\it near history} involving
sources up to ${\mathcal O}(1)$ domain traversal times into the past,
plus
(c) a (very smooth) {\it far history} comprising all waves emitted before the near history.
The local part is computed directly via high-order quadrature.
A naive spatial Fourier transform for (b) plus (c)
would be both slowly converging and arbitrarily oscillatory as time progresses.
Yet in (b) the oscillations are controlled, so
we use the recent truncated windowed Fourier projection (TK-WFP) method to give
rapid convergence.
For (c)---present due to the weak Huygens' principle---%
we exploit a new large-time sum-of-exponentials approximation of the
free-space wave kernel.
Numerical examples with up to a million sources and targets,
a domain of $300\times 300$ wavelengths, and 6-digit accuracy, 
show an acceleration of five orders of magnitude relative to direct evaluation.
\end{abstract}

\begin{keyword}
wave equation, free-space, hyperbolic potentials, high-order, truncated kernel, Fourier methods
\end{keyword}


\section{Introduction}\label{sec:intro}

The rapid evaluation of free-space hyperbolic potentials---%
integral representations of the solution to the wave equation---is key to the 
development of geometrically flexible and high order accurate methods
for time domain wave scattering problems that arise in
acoustics \cite{kaltenbacher2018computational,Takahashi2014},  
electromagnetics \cite{CMS,liu18},
and elastodynamics \cite{takahashi03}.
Such methods are not subject to grid-based dispersion errors, can avoid restrictive stability-based CFL conditions, and do not require radiation boundary conditions. Moreover, for homogeneous problems (ones without a volumetric source), they require a discretization of the boundary alone.
The naive computation of
such hyperbolic potentials, however, is extremely expensive since the 
representation is global in both space and time.
Evolving the solution in the Fourier domain overcomes the history dependence,
as we will see below, but encounters two obstacles that make it appear 
impractical. First, the solution is non-smooth, so that the (spatial) Fourier
transform is slowly decaying, and, second, the Fourier 
transform becomes more and more oscillatory as time increases.
In the present paper, we design an
algorithm that circumvents both difficulties, making use of the
truncated
windowed Fourier projection method \cite{wfp2025,tkwfp3d}
and a new, 
non-oscillatory integral
representation of the
2D wave kernel for large time.

Our concern here is the evaluation of solutions
to the 2D scalar wave equation:
\beq\label{eq:wave2d}
\begin{cases}
\pt^2u - \Delta u = f(\xv,t), & \xv\in\mathbb R^{2}, \ t \in (0,T], \\
u(\xv,0) = \pt u(\xv,0) = 0, & \xv\in\mathbb R^{2}. 
\end{cases}
\eeq
whose exact solution is
\beq\label{eq:solnRep}
u(\xv,t)  = \int_0^t \int_{\mathbb R^2} G(\xv - \yv, t- \tau) f(\yv,\tau)d\yv d\tau, \qquad t \in (0,T], 
\eeq
where
$G(\xv,t)$ is the Green's function given by
\beq\label{eq:Green}
G(\xv,t) = \frac{H(t - \abs{\xv})}{2\pi\sqrt{t^2 - |\xv|^2}},
\qquad t>0, \ \xv\in\mathbb R^2,
\eeq
where $H$ is the usual Heaviside step function, and $|\cdot|$ denotes the Euclidean norm.
In order to focus on the development of a fast algorithm, we omit 
discussion of discretization and quadrature, and consider as our model
problem the case where $f$ is a sum of $M$ point sources,
\beq\label{eq:pointSources}
f(\xv,t) = \sum_{j=1}^M\delta(\xv-\yv_j)\sigma_j(t),\qquad \xv\in \mathbb R^2, 
\eeq
where $\delta(\xv)$ represents the two-dimensional Dirac delta distribution, and the time signatures $\sigma_j$ are smooth but possibly wide-band functions on $t\in\R$, such that $\sigma_j(t) = 0$ for $t\le 0$.
We restrict source locations $\yv_j$  to a computational domain $B := [-1,1]^2$. When the sources are distributed throughout
the domain, this calculation can be viewed as
a discretized {\it volume potential}.
When the sources are restricted to a boundary, 
it can be viewed as a discretized {\it single 
  layer potential}.
In either case, temporal discretization of \eqref{eq:solnRep}
on a time grid $t_n = n \dt$ with $\Nt$ total time steps would
result in a calculation of the form
\beq\label{eq:solnRep_discrete}
u(\xv_i,t_n) \approx
\sum_{l=1}^{n} \sum_{j=1}^M 
G(\xv_i - \yv_j, t_n- t_l) \sigma_j(t_l),
\qquad
i=1,\dots,\Nx, \;  n = 1,\dots,\Nt,
\eeq
for given target points $\{ \xv_i \}_{i=1}^{\Nx}$.
We assume that $\dt$ is sufficiently small to resolve all signatures $\sigma_j$
to the desired precision, and defer the issue of singular quadrature.
Yet by inspection of~\eqref{eq:solnRep_discrete}, it is clear that
direct evaluation requires $\Oh{M\Nx\Nt^2}$ operations,
which is quadratic in both space and time.
While we do not seek to review
the literature in detail, 
a brief overview of existing methods to cope with this computational burden is provided in section \ref{priorwork}.

From an analytic perspective, it is natural to consider Fourier 
analysis and the spectral representation of the Green's function.
For this, we define the spatial 2D Fourier transform and its inverse by
\beq
\label{eq:fourierdef}
\hat u (\kv,t) = \int_{\mathbb R^2} u(\xv,t) e^{i\kv\cdot \xv}d\xv, \quad \text{and} \quad u (\xv,t) =\frac{1}{(2\pi)^2} \int_{\mathbb R^2} \hat u(\kv,t) e^{-i\kv\cdot \xv}d\kv.
\eeq
It is well known, and straightforward to derive, that
the spectral form of the Green's function is
\beq\label{eq:spectralGreen}
G(\xv, t) = \frac{1}{(2\pi)^2} \int_{\mathbb R^2} \frac{\sin \kappa t}{\kappa} e^{-i\kv\cdot \xv }d\kv, \qquad \kappa := |\kv|, \quad t>0,
\eeq
and that
\beq
\label{eq:uexactF}
u(\xv,t) = \frac{1}{(2\pi)^2} 
\int_{\mathbb{R}^2}  e^{-i \kv \cdot \xv}
\int_0^t  
 \frac{\sin \kappa (t-\tau)}{\kappa} \sighat(\kv,\tau) d\tau
d\kv,
\eeq
the spectral source function being
\beq\label{eq:sighat}
\sighat(\kv,t) = \sum_{j = 1}^M  \sigma_j(t) e^{i \kv \cdot \yv_j}.
\eeq
There are three difficulties in making such a Fourier-based solution practical.
First, since the source is spatially non-smooth,
the integrand in \eqref{eq:uexactF} decays slowly as $|\kv|\to\infty$.
Second, the spectral form of the Green's function becomes more and more oscillatory with time, requiring a finer and finer discretization of the Fourier transform. Third, the representation is still dependent on the full space-time history of the source strengths $\sigma_j(t)$. 
In one and three dimensions, we showed how to overcome these obstacles
in \cite{wfp2025,tkwfp3d}.
In two dimensions, however, the lack of a strong Huygens' principle adds 
a significant complication. Waves linger within the domain for the
entire simulation time, and addressing this issue is one of the main technical 
contributions of the present work.

\subsection{Brief review of existing methods} \label{priorwork}
Fourier-based fast algorithms to overcome history-dependence have been 
developed for both the heat equation and the Schr\"{o}dinger equation
\cite{Greengard2000,greengard1990cpam,Kaye2022}. 
For the heat equation, the spectral representation of the Green's function
is rapidly decaying and non-oscillatory and the main difficulties involve
the short time behavior of volume and layer potentials.
The issues in the Schr\"{o}dinger case are closer, particularly in the need to 
avoid high frequency oscillations in the 
Fourier transform over long simulation times.
The algorithm of~\cite{Kaye2022} overcomes this 
by contour deformation of the Fourier transform into the complexified Fourier 
domain. Such a deformation does not appear to be feasible for the 
scalar wave equation.

In three dimensions,
the fast algorithm of \cite{tkwfp3d} 
consists of three key steps: 
(1) replacing the free space Green's function by a truncated Green's function
which is identical over the domain of interest,
(2) applying a smooth splitting of the solution into a time-local part,
evaluated directly, plus a smooth history part evaluated using Fourier methods,
and (3) using the non-uniform fast Fourier transform to deal with
irregularly spaced data. Critically, the
{\it spatial} truncation of the Green's function leads 
to {\it temporal} truncation as well, avoiding the oscillatory behavior of the 
wave kernel for long simulation times. 

Other methods to evaluate \eqref{eq:solnRep} include frequency-time hybrid (FTH) approaches, convolution quadrature (CQ) methods, and plane-wave-time-domain (PWTD) algorithms.
FTH methods approximate the inverse Fourier transform in time from 
a large set of \textit{independent} frequency-domain solutions, computed via
Helmholtz boundary integral solvers.
Anderson, Bruno, and Lyon,
for example, presented an FTH method~\cite{Anderson2020}
where they split the incident field into compactly supported time windows, from which they reconstruct the solution at any time.
One challenge in FTH methods is tackling the
resonance poles in the complex frequency plane, 
associated with trapped modes which dominate the late-time dynamics.
Since they lie arbitrarily near the real axis, such poles complicate the Fourier
inversion integral.
Wilber et al.\ \cite{Wilber2025} address this via an imaginary shift of the
contour in their fast sinc transform FTH algorithm.
Bruno and Santana~\cite{Bruno2025} present a FTH variant that subtracts off nearby poles
(handled using an asymptotic expansion) to leave a smoother inverse Fourier transform.
A tougher challenge is that
in the resonant case, iterative Helmholtz solvers have an iteration count
growing linearly (in 2D) with frequency \cite{marchand22}.

Convolution quadrature (CQ) methods
use quadrature approximations of convolutions performed in the Laplace transform plane
\cite{lubich94}. 
This requires a large set of independent solutions (usually found by boundary
integral methods) at complex Helmholtz frequencies,
thus is similar in spirit to FTH with contour deformation.
At least at low frequencies, CQ
methods reduce the time complexity from $\Oh{\Nt^2}$ to $\Oh{\Nt\log\Nt}$ or
$\Oh{\Nt\log^2\Nt}$.
See, for example, Monegato and Scuderi's method \cite{Monegato2013}, and Banjai, López-Fernández, and Schädle's Runge-Kutta-based convolution quadrature with oblivious quadrature~\cite{Banjai2016}. Such methods can, however, become inaccurate with low-regularity data, 
and suffer similar challenges as FTH with high frequencies and poles \cite{Betcke17}.

In plane-wave time domain (PWTD) methods, far-field
interactions are approximated using plane wave expansions, 
with sources and targets grouped hierarchically to accelerate translation 
and evaluation, in a similar style to the high-frequency fast multipole method \cite{rokhlin_2d}.
Asymptotically optimal schemes have been developed 
for the two-dimensional setting in \cite{Lu2000, Lu2004, Lu2004_2}.
They are optimal for both pure boundary value problems and for problems with adaptive 
volumetric discretizations, but the implementation is quite complex and 
the associated constants can be large.
Closer to our approach is the Fourier-based time-domain adaptive integral method of \cite{Yilmaz2004}, which exploits the convolution structure of the interaction {\it in space and time} (but does not try to evolve the spectral representation of the spatial Fourier transform as we do here).

More standard than any of the above, of course, is to compute
solutions of \eqref{eq:wave2d} using direct temporal and spatial 
discretization through finite difference (FDTD) or finite element methods, with radiation boundary conditions imposed at the computational boundary. Although exact radiation boundary conditions are non-local in space and time, many
local approximations of such conditions have been developed, such as those
of Engquist and Majda~\cite{Engquist1977}, Bayliss and Turkel~\cite{Bayliss1980}, and Higdon~\cite{Higdon1990}. Another approach to radiation conditions is the 
gradual addition of dissipative terms to the governing partial differential
equation, such as the perfectly matched layer approach of~\cite{Berenger1994} 
or the absorbing region method of~\cite{Israeli1981}. The double absorbing boundary (DAB) method of Hagstrom et al. in~\cite{DAB_Hagstrom2014} approximates the boundary
data corresponding to the free-space solution in a 
thin layer surrounding an artificial boundary by introducing auxiliary variables and 
solving a set of auxiliary PDEs in that layer. Complete radiation boundary conditions (CRBCs) due to Hagstrom and Warburton~\cite{CRBC1_Hagstrom2004, CRBC2_Hagstrom2009} avoid introducing this thin layer by replacing normal derivatives in the auxiliary PDEs with temporal derivatives and solving surface PDEs directly on the artificial (rectangular) boundary.
Both DABs and CRBCs can be coupled to standard discretizations in the 
interior of the rectangular domains. For circular boundaries,
an exact radiation condition is described in~\cite{Alpert2000, Alpert2002}, together with
a fast algorithm for computing the Dirichlet-to-Neumann map. 
Also worth noting are time-dependent ``phase space filters''
\cite{SOFFER2009, SOFFER2007, SofferStucchio}
and ``global discrete artificial boundary conditions''
\cite{tsynkov01}.
These permit the simulation of radiation boundary conditions with precision-dependent control.

\subsection{Outline of our approach}

Because of the lack of a strong Huygens' principle in 2D, we require 
more elaborate machinery than in the 3D setting, in order to handle the
wake left behind by sources in the distant past.
For this, we split the solution into  {\it three} parts:
a local component, a near-history component, and a far-history component:
\beq\label{eq:split1}
u(\xv,t) = \ul(\xv,t) + \unh(\xv,t) + \ufh(\xv,t). 
\eeq
The local part $\ul$ spans a time interval within a few time steps of the current time; $\unh$ spans a
time interval of the order of one passage time (the time required for a wave to traverse the
computational domain $B = [-1,1]^2$),
and $\ufh$ involves solution values from the temporal cut-off point of the near history all the way back to the initial time $t=0$.

To separate $\ul$ from the history part $\uh = \unh + \ufh$, we use the Windowed Fourier Projection (WFP) method, introduced in~\cite{wfp2025}.
Essentially an ``Ewald split'' for the wave equation,
this applies a blending function $\phi$
to partition the solution so that $\ul$ is non-smooth but local in time (and hence space),
while $\uh$ is nearly as smooth 
as the source signatures $\sigma_j(t)$ in \eqref{eq:sighat} allow.
The function $\phi$ is defined so that $\phi(t)=0$ for $t\leq 0$, and $\phi(t) = 1$ for $t\geq \delta$, where $\delta = \Oh\dt$. 
Adjusting the width $\delta$ of the blending function controls the spatial smoothness of the history part
or, more precisely, the bandwidth extension beyond the intrinsic bandwidth of the 
source functions in \eqref{eq:sighat}.
The larger $\delta$ is, the smaller the bandwidth extension and the faster the decay of the Fourier transform. 

While the WFP method enforces rapid decay of the Fourier transform, 
it does not control the oscillatory behavior (with respect to $\kv$)
of the spectral wave kernel as time increases.
We accomplish that here though the further split of the history part into
$\unh$, which has controlled oscillations (and thus can be discretized in $\kv$
out to around the signature bandwidth, as in \cite{wfp2025,tkwfp3d}),
plus $\ufh$, which is spatially much smoother, being a sum of the
weak Huygens' algebraic tails of $G$.
To represent $\ufh$ we again truncate $G$
in a way that has no effect within the computational domain $B$,
but now using a purely spatial blending function of width $\Oh{1}$.
As a result, $\ufh$ is discretized in the Fourier domain with a modest $\Oh{1}$ number of
quadrature points, {\it regardless of the signature bandwidth or final simulation time}.
However, unlike $\unh$ (as in \cite{wfp2025,tkwfp3d}),
each Fourier component of $\ufh$ no longer obeys a 2nd-order Duhamel relation,
necessitating a
temporal sum-of-poles approximation in which each term
does obey a Duhamel relation.

This paper is organized as follows. In Section~\ref{sec:prelim}, we introduce the tools needed
for the development of the 2D Truncated Kernel Windowed Fourier Projection (TK-WFP) algorithm. 
Section~\ref{sec:partition} covers the partition of the solution into local, near-history, and
far-history parts. We discuss the approximation and computation of each part in Sections~\ref{sec:near},~\ref{sec:far}, and~\ref{sec:local}, respectively. 
In Section~\ref{sec:results}, we verify the performance of the 2D TK-WFP algorithm using numerical examples featuring up to a million sources with frequency bandwidth corresponding to
300 wavelengths per side of the square domain.
We end with concluding remarks in Section~\ref{sec:conclusion}. 

\section{Components of the method}\label{sec:prelim}

Key to the development of our method is the smooth blending function $\phi$. 
We begin with its precise definition and a summary of its properties. We then
describe the truncated kernel splits,
and a sum-of-exponentials approximation of $1/\sqrt{t^2 - r^2}$,
which will be used for the rapid evaluation of the Fourier transform of the far history.

\subsection{Blending function}
For both temporal and spatial truncation we will use
the continuous blending function $\phi$ introduced in the original 1D WFP method~\cite{wfp2025}.
Given a width parameter $\delta>0$, this is defined as
\beq\label{eq:KBblending}
\phi_{\delta}(t) := 
\int_0^t\phi_{\delta}'(\tau)d\tau,
\quad \text{where}
\;
\phi_{\delta}'(t):= \begin{cases}\frac{b}{\delta\sinh b}I_0\left(b\sqrt{1 - (2t/\delta - 1)^2}\right), & 0\leq t\leq \delta, \\
  0, &
  \mbox{otherwise.}
\end{cases}
\eeq
Here, $I_0$ is the zeroth order modified Bessel function of the first kind and $\phi_{\delta}'$ can be viewed as a 
scaled and shifted Kaiser--Bessel bump function $I_0(b\sqrt{1 - t^2})$, $t\in[-1,1]$,
with unit $L^1$-norm.
Thus $\phi_\delta(t) = 0$ for $t\le 0$, while $\phi_\delta(t) = 1$ for $t\ge \delta$.
The parameter $b$ is a precision-dependent shape parameter which controls the bandwidth of $\phi_{\delta}$.
Given $\epsilon>0$, $\epsilon\ll 1$, and fixing $b = \ln(1/\epsilon)$,
the functions $\phi_{\delta}$, and $\phi_{\delta}'$ 
are \textit{numerically} smooth; 
that is, $\phi_{\delta}'(t)$ is smooth except for
jumps
of size $\Oh{e^{-b}}=\Oh\epsilon$ at $t = 0$ and $t = \delta$. The Fourier transform of 
$\phi_{\delta}'$ is available analytically, and is given by
\beq
\widehat{\phi_{\delta}'}(\omega) := \int_{\mathbb R} \phi_{\delta}'(t)e^{i\omega t}dt = \frac{be^{-i\delta\omega/2}}{\sinh b}\sinc\sqrt{\left(\frac{\delta\omega}{2}\right)^2 - b^2}, 
\eeq
where $\sinc\, z := \sin(z)/z$ for $z\neq 0$ and $1$ otherwise.
The bump function $\phi_\delta'$ is $\epsilon$-bandlimited to
$[-2b/\delta, 2b/\delta]$ in the sense that,
for any $\theta>1$,
\beq\label{eq:windowBandlimit}
\abs{\widehat{\phi_{\delta}'}(\omega)}< \frac{4b\theta\epsilon}{\delta\abs{\omega}},\quad
\text{for all }
\abs{\omega}\geq \frac{2b}{\delta\sqrt{1 - \theta^{-2}}}. 
\eeq
For the temporal (local-history) split,
the property~\eqref{eq:windowBandlimit}, proved in \cite[App.~A]{tkwfp3d},
will be needed to establish the rapid decay of the near-history
Fourier data beyond a cut-off wavenumber.
For the temporal blending from
local to near-history, and near-history to far-history,
the width will be small: $\delta = \Oh{\Delta t}$ (see Sec.~\ref{s:delta}).
However, for the spatial truncation of the far-history,
the width will be $\Oh{1}$, i.e., larger.

\subsection{Far history spatial truncation and temporal partition}
\label{s:split}

The far history component will use the
following {\it spatial} truncation of the 2D Green's function,
\beq\label{eq:TKdef}
\calG_A (\xv,t) = \phiD(A - \abs{\xv})G(\xv,t), \qquad \xv\in\mathbb R^2, t> 0,
\eeq
where $\phiD$ is defined as in \eqref{eq:KBblending}, but with a larger radial blending
width parameter $\Delta = \Oh{1}$.
Then $\calG_A$ vanishes for all $|\xv|>A$, while equalling the true $G$ \eqref{eq:Green}
throughout the ball $|\xv| \le A-\Delta$.
Thus by choosing $A - \Delta$ at least the largest distance
between any source and target in $B$, i.e. $A \ge 2\sqrt{2}+\Delta$,
the solution representation in
\eqref{eq:solnRep} for $\xv\in B$ is
unchanged by replacing $G$ by $\calG_A$.
For efficiency, we use the smallest such value,
\beq
A = 2\sqrt{2} + \Delta.
\label{A}
\eeq
By the Paley--Wiener theorem \cite{steinweissbook},
the spatial truncation of $\calG_A$ controls the oscillation rate of the integrand
{\it in the spatial Fourier domain},
so that using a Nyquist-spaced trapezoidal quadrature for the inverse
Fourier transform requires a 
$\kv$ grid spacing of only $\Oh{1}$ to achieve spectral accuracy, for all time.
Conversely, but distinctly from this,
the spatial blending width $\Delta$ will control the smoothness of $\ufh$ in space,
and thus the maximum $\kappa=|\kv|$ that is needed in this quadrature
to achieve an $\Oh{\epsilon}$ error.
\begin{remark}
The spatial blending function used in \eqref{eq:TKdef} will always be referred to as $\phiD$.
With a slight abuse of notation, from now on
we will drop the subscript in the (narrow) temporal
blending function \eqref{eq:KBblending} and denote $\phi_\delta$ simply by $\phi$.
\end{remark}

\begin{figure}[t]  
  \includegraphics[width=1.0\textwidth]{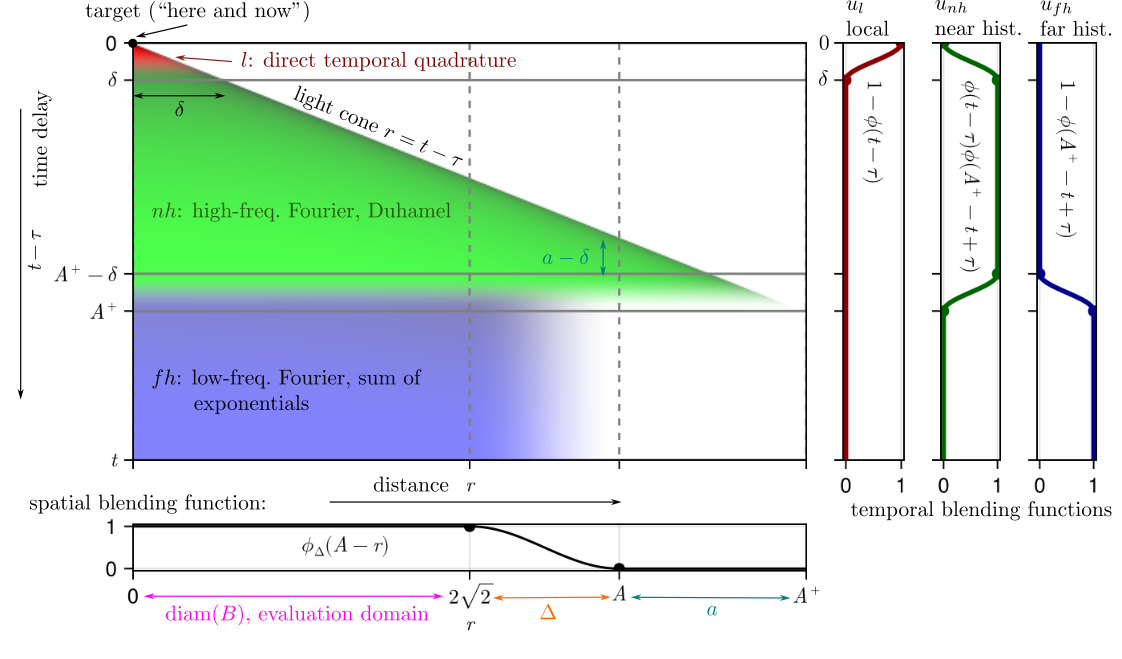} 
  \vspace{-5ex}
  \caption{Space-time diagram showing the influence at the target point $(\xv,t)$ of the
    three components in our method: local (shaded red) given by \eqref{eq:ul}, near history (green) given
    by \eqref{eq:unh}, and far history (blue) given by \eqref{eq:ufh}. (See online for color.)
    The color gradation indicates the blending (smooth multiplication) applied to the
    2D free-space Green's function $G(\xv-\yv,t-\tau)$,
    where $r = |\xv-\yv|$ is the distance from a
    source $\yv$, and $t-\tau$ is the time delay (increasing downwards into the past).
    White indicates zero.
    The darkness hints at the value of $G$, showing the $-1/2$ power singularity along the
    light cone.
    The time axis is shared with the three temporal blending functions to the right.
    The $r$ axis is shared with the spatial blending function for the far history only (bottom).
    The parameters $\delta$, $\Delta$, $a$, $A$ and $A^+$ are explained in Section~\ref{s:split}.
    Blending functions are shown with unrealistically small $b$ for better visualization.
  }
	\label{fig:partition} 
\end{figure}

We may now define the three components in the solution representation \eqref{eq:split1},
which uses the following $\delta$-scale temporal blending both to split the local from the
near-history (over $0<t-\tau<\delta$),
and also the near-history from the far-history (over $\Ap-\delta < t-\tau < \Ap$),
where the near-far split parameter is
\beq
\Ap = A + a
\label{Ap}
\eeq
for a parameter $a>0$ of size $\Oh{1}$ that
will enable an efficient sum-of-exponentials approximation of $\hat\calG_A(\kv,t)$.
The truncated kernel windowed Fourier projection (2D TK-WFP) method then
consists of using distinct fast algorithms for the evaluation of each solution component:
\begin{subequations}\label{eq:split}
\beq\label{eq:ul}
\ul(\xv,t) = \sum_{j = 1}^M\int_{t-\delta}^t G(\xv - \yv_j, t- \tau) \sigma_j(\tau)[1 - \phi(t - \tau)] d\tau,  
\eeq
\beq\label{eq:unh}
\unh(\xv,t) =  \sum_{j = 1}^M\int_{t-\Ap}^t G(\xv - \yv_j, t- \tau)  \sigma_j(\tau)\phi(t - \tau)\phi(\tau - t + \Ap) d\tau, 
\eeq
\beq\label{eq:ufh}
\ufh(\xv,t) =  \sum_{j = 1}^M\int_{0}^{t - \Ap + \delta} \calG_A(\xv - \yv_j, t- \tau)  \sigma_j(\tau)[1 - \phi(\tau - t + \Ap)] d\tau.
\eeq
\end{subequations}
Recall that, for all $\xv$ and $\yv_j$
in the solution domain $B$, $\calG_A$ is equivalent to $G$, so that \eqref{eq:ufh}
is valid.
Figure~\ref{fig:partition} illustrates the partition.

The local part is computed directly.
It spans a small number of time steps and requires
only a quadrature rule in time for each source-target pair with spatial separation $\le \delta$,
with special care taken when their separation is small compared to the time step.

The history parts, on the other hand, are computed by spatial inverse Fourier transforms:
since their Green's functions have strict $\Oh{1}$ radial supports (namely $\Ap$ for $\unh$ and
$A$ for $\ufh$), a fixed $\kv$ quadrature grid will be sufficient for both.
In 2D, the radial Fourier transform $F(\kappa)$
of a radial function $f(r)$ is the Hankel transform
\[
F(\kappa) = 2 \pi \int_0^\infty J_0(\kappa r) f(r) r \, dr,
\]
where $J_0$ is the zeroth-order Bessel function of the first kind.
Thus the Fourier transform of $\calG_A(\xv,t)$ is (recall $\kappa=|\kv|$),
\beq\label{eq:spectralTK}
\hat\calG_A(\kv,t) = \int_0^{\min(A,t)}\frac{rJ_0(\kappa r)}{\sqrt{t^2 - r^2}}\phiD(A - r)dr,
\quad t>0.
\eeq
By contrast, recall that the free-space kernel
has the Fourier transform
$$\hat G(\kv,t) = \int_0^t \frac{rJ_0(\kappa r)}{\sqrt{t^2 - r^2}}dr = \frac{\sin\kappa t}{\kappa},\quad t>0,
$$
and thus has unbounded oscillation rate in $\kappa$ at long times.
It is the truncation in the upper limit of integration in \eqref{eq:spectralTK} that allows the far
history to be well represented with a fixed $\kv$-grid for all time.

This use of the truncated kernel $\calG_A$ in $\ufh$, however, introduces a new difficulty.
Namely, while Euler's formula
$\sin \kappa t = (e^{i \kappa t} - e^{- i \kappa t})/2i$, 
leads to a simple Duhamel-type recurrence
for updating in time the spectral representation of $G$, and hence $\widehat{\unh}(\kv,t)$,
we will need a more elaborate 
method for efficiently updating $\widehat{\ufh}(\kv,t)$.

\subsection{Sum-of-exponentials approximation of the wave kernel}
\label{s:SOE}

In order to construct an efficient recurrence for the Fourier coefficients of the far history,
we start with the identity (see \cite[(6.611.4)]{GR8} with $\nu=0$)
\beq\label{eq:LaplaceIntegral}
\frac{1}{\sqrt{t^2-r^2}} = \int_0^{\infty} e^{-\lambda t} I_0(r\lambda)d\lambda,
\qquad t>r\ge 0,
\eeq
expressing the wave kernel (as in \eqref{eq:Green}) inside the light cone
as a Laplace transform.
We apply a composite quadrature to this integral,
using an $N_g$-node Gauss--Legendre rule in each of the $n$ ``panels'' (intervals)
\[
\left[0,\lambdam/2^n \right], 
\left[\lambdam/2^n,\lambdam/2^{n-1}\right], \left[\lambdam/2^{n-1},\lambdam/2^{n-2}\right], \dots, \left[\lambdam/2,\lambdam\right].
\]
We denote the entire set of resulting nodes and weights 
by $\lambda_l$ and $q_l$, respectively, indexed by $l = 1, \dots, \Nlam$, where $\Nlam = nN_g$.
The quadrature approximation is thus
\beq\label{eq:LaplaceApprox}
\frac{1}{\sqrt{t^2 -r^2}} = \sum_{l = 1}^{\Nlam}q_l I_0(r\lambda_l)e^{-\lambda_l t} + \Oh{\tilde\epsilon},
\qquad
r\in[0,A], \; t\in[\Ap-\delta,T],
\eeq
and we will show how to choose the three parameters $\lambdam$, $n$, and $N_g$ such that
this holds for a small desired tolerance $\tilde\epsilon$, uniformly over the
required $(r,t)$ domain.
Here the maximum radius needed is $A$
because this is the radial support of $\calG_A$ in \eqref{eq:TKdef}.
The $t$ parameter in \eqref{eq:LaplaceApprox} will be substituted by the time delay $t-\tau$
in the far history \eqref{eq:ufh}; this explains why the minimum time delay is
$\Ap-\delta$ while the maximum is the simulation end time $T$.

Firstly, 
the truncation parameter $\lambdam$ may be set by considering the
exponential decay rate of the integrand in \eqref{eq:LaplaceIntegral}.
Up to a weak algebraic factor, $I_0(r\lambda)\sim e^{r\lambda}$ as 
$\lambda\rightarrow\infty$,
so the integrand in \eqref{eq:LaplaceIntegral} behaves like
$e^{-\lambda(t - r)}\leq e^{-\lambda(\Ap - A)} = e^{-\lambda(a - \delta)}$, noting \eqref{Ap}.
(This minimum far-history separation $a-\delta$ from the light cone is shown
in Figure~\ref{fig:partition}.)
Since $\delta\ll1$, setting $a=1$ means that the truncation error is
of order $e^{-\lambdam}$, so that $\lambdam=36$ makes this close to double precision accuracy.

Secondly, the full range of decay rates must be accurately integrated, which is
guaranteed by the dyadically graded grid of quadrature panels. 
Consider the most rapidly-decaying case: for $r=0$ and $t=T$ the integrand
is $\sim e^{-\lambda T}$, so that the first panel $[0,\lambdam/2^n]$ must 
be of size $\Oh{1/T}$ to accurately integrate this. 
This demands that $n \approx \log_2 (\lambdam T) = \log_2 (36 T)$.
We see only logarithmic growth with respect to $T$. In practice we set $n=20$, allowing
$T$ up to about $3\times 10^4$, about 15,000 passage times across the domain.

Thirdly, one must set $N_g$, the number of nodes per panel.
Since the integrand is analytic (in fact entire), we have exponential convergence
in $N_g$ (see, e.g., \cite[Thm.~19.3]{ATAP}).
We find that $N_g=32$ is sufficient for close to double precision accuracy across the
desired $(r,t)$ domain in this work.
Thus the sum of exponentials has $N_\lambda = 640$ terms.
A rigorous bound of this dyadic quadrature scheme would be possible; we leave this
for future work. See \cite{Greengard2000} for a related scheme with analysis.

\section{Representations of the three solution components}\label{sec:partition}

Let us now revisit the smooth partition in \eqref{eq:split1} into local, near history, and far history 
parts, defined in~\eqref{eq:split} using the blending function $\phi$ in~\eqref{eq:KBblending},
the free-space wave kernel $G(\xv,t)$ from \eqref{eq:Green} and the truncated kernel
$\calG_A(\xv,t)$ in \eqref{eq:TKdef}.

The local part of the solution in \eqref{eq:ul} takes the form
\beq\label{eq:ul1}
\ul(\xv,t) = \frac{1}{2\pi}\sum_{j\in\calNd(\xv)}\int_{t - \delta}^{t - r_j}\frac{\sigma_j(\tau)[1 - \phi(t - \tau)]}{\sqrt{(t - \tau)^2 - r_j^2}}d\tau, 
\eeq
where
$r_j: = \abs{\xv - \yv_j}$, and $\calNd(\xv) = \{\ j \ | \ 0<r_j<\delta, \ j = 1, 2,\dots, M\}$ represents the indices of
non-coincident
sources within a ball of radius $\delta$ centered at $\xv\in B$. 
Recall from~\eqref{eq:ul1} that 
$\ul$ spans a time interval of only a few time steps, 
since $\delta = \Oh{\Delta t}$, and that  only a small number of sources are 
located within a ball of radius $\delta$ centered at any given target.
While the integral requires some care because of the near singularity for small $r$,
suitable quadrature rules for the local part can be precomputed for a fixed $\Delta t$
and stored in a sparse matrix that is valid for all time steps.
We defer a detailed discussion to Section~\ref{sec:local}. 

For time-stepping the history parts of the solution $\unh,\ufh$,
we turn to the Fourier domain. 
The near history \eqref{eq:unh} has the inverse Fourier transform
representation, by analogy with \eqref{eq:uexactF},
\beq\label{eq:uh_spectral}
\unh(\xv,t)  = \frac{1}{(2\pi)^2} \int_{\mathbb R^{2}} \alpha(\kv,t) e^{-i\kv\cdot \xv}d\kv, 
\eeq
with
\beq\label{eq:alphak}
\alpha(\kv,t) = \int_{t - \Ap}^{t}\frac{\sin \kappa (t - \tau)}{\kappa}
\phi(t - \tau)\phi(\tau - t + \Ap)\sighat(\kv,\tau)\, d\tau, 
\eeq 
recalling the source spatial Fourier transform $S(\kv,\tau)$ defined in \eqref{eq:sighat}.
Here, the zero-frequency term $\alpha({\bf 0},t)$ is taken in the sense of the limit
$\kappa\to0$.
Because \eqref{eq:alphak} is its temporal convolution with a blended sine
function, a 2nd-order Duhamel time-step update for $\alpha$ is possible,
driven only by $S$ in the two narrow transition intervals,
just as in our 3D method \cite{tkwfp3d}.
We give the details and the resulting numerical approximation of $\unh$ in Section~\ref{sec:near}.

We similarly express $\ufh$ in \eqref{eq:ufh} as the inverse Fourier transform
\beq\label{eq:ufh2}
\ufh(\xv,t) = \frac{1}{(2\pi)^2} \int_{\mathbb{R}^2}\alpha_F(\kv,t)e^{-i\kv\cdot\xv} d\kv,
\eeq
where, using the Fourier transform of the radially-truncated Green's function, $\hat\calG_A(\kv,t)$
in \eqref{eq:spectralTK},
\beq
\begin{split}
  \alphaF(\kv,t) = &
  \int_0^{t - \Ap + \delta} \hat\calG_A(\kv,t-\tau) [1 - \phi(\tau - t + \Ap)]
  \sighat(\kv,\tau) \,  d\tau
  \\
  = & \int_0^{t - \Ap + \delta}  [1 - \phi(\tau - t + \Ap)] \sighat(\kv,\tau)
\int_0^A\frac{rJ_0(\kappa r)\phi_{\Delta}(A - r)}{\sqrt{(t - \tau)^2 - r^2}}
dr \, d\tau.
  \label{eq:alphaF}
\end{split}
\eeq
Note that the upper limit of the inner integral in \eqref{eq:alphaF} is
$A$ rather than $\min(A,t - \tau)$ as in \eqref{eq:spectralTK};
this follows since $t-\tau \ge \Ap-\delta > A$ for the far history.
The inner integral is the Hankel transform on $r\in[0, A]$
of the function $\phi_{\Delta}(A - r)/\sqrt{(t - \tau)^2 - r^2}$, at fixed delay $t - \tau\in[\Ap-\delta,t]$.
Recalling \eqref{Ap}, $t-\tau \ge a-\delta = \Oh{1}$, so that the denominator is smooth
on an $\Oh{1}$ radial scale; the same is true for the numerator because
the blending function has width $\Delta=\Oh{1}$.
Thus the Hankel transform decays rapidly in $\kappa$,
and may be truncated with close to machine precision error at a
moderate maximum $\kappa$ that is {\it independent}
of the wave frequency (signature bandwidth).
We demonstrate this, and give further details on the 
approximation of $\ufh$ and the computation of $\alpha_F(\kv,t)$, in Section~\ref{sec:far}. 

\section{Evaluation of the near history}\label{sec:near}

Suppose now that $\alpha(\kv,t)$ is available, and
we discretize the Fourier representation of the near history
part in \eqref{eq:uh_spectral} using the (infinite) tensor product trapezoidal rule with 
grid spacing $\dk$:
\beq\label{eq:uhtrap}
\unh(\xv,t) \approx \left(\frac{\dk}{2\pi}\right)^2 \sum_{\nv \in\mathbb Z^2}\alpha(\nv\dk,t)e^{-i\nv\dk\cdot \xv}, \qquad \xv \in B,
\eeq
which may be interpreted as a Fourier {\it series} with spatial period $2\pi/\dk$.
Remarkably, for $\dk\leq 2\pi/(\Ap + 2)$ this expression 
is {\it exact}.
This follows from the Poisson summation formula \cite[Ch.~VII, Cor.~2.6]{steinweissbook},
\beq \label{PSFerror}
\unh(\xv,t) -
\left(\frac{\dk}{2\pi}\right)^2\sum_{\nv \in\mathbb Z^2} \alpha(\nv\dk,t)e^{-i\nv\dk\cdot \xv} 
= - \sum_{\mv\in\mathbb Z^2\backslash \{0\}} \unh\left(\xv + \frac{2\pi}{\dk} \mv,t\right),
\eeq
which expresses the quadrature error (left side) as a sum over a punctured lattice
of images with spacing $2\pi/\dk$ (right side).
From the definition \eqref{eq:unh} of $\unh$ and the unit propagation
speed, the spatial support of $\unh$ lies within $(-\Ap - 1,\Ap + 1)^2$.
If the lattice spacing is large enough, no translation of this support can fall within $B$.
Figure~\ref{fig:imagelatt} explains this geometric constraint.
See \cite[Prop.~3.1]{tkwfp3d} for a formal proof in the 3D case.
For exact quadrature,
highest efficiency results at the largest Fourier grid spacing, namely
\beq\label{eq:dk}
\dk = \frac{2\pi}{\Ap + 2},
\eeq
which will be around $0.92$ in our numerical tests in Section~\ref{sec:results}.

\begin{figure}[t]  
    \centering
  \includegraphics[width=0.8\textwidth]{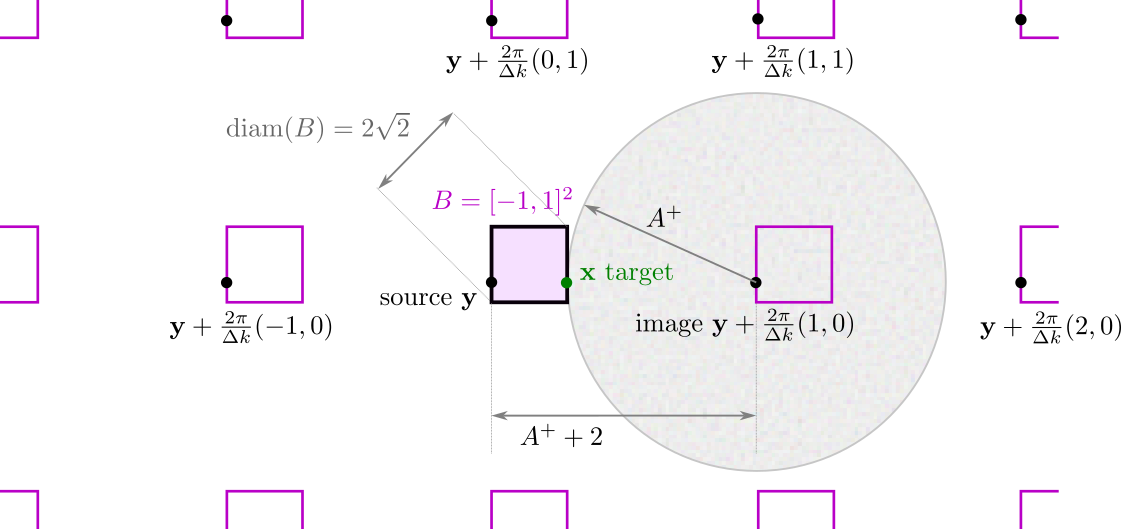}
  \vspace{-1ex}
  \caption{The geometric constraint on the maximum $\Delta k$ grid size for which the
    trapezoid rule (equispaced infinite grid) quadrature is exact for the inverse Fourier transform
    evaluation of the near history $\unh$; see Sec.~\ref{sec:near}.
    By the Poisson summation formula, quadrature (aliasing) creates
    an infinite lattice of sources, of which all but the central one are erroneous.
    $\yv$ is such a source on the boundary of the computational domain $B$.
    By Huygens' principle, the near history lives in an open disk of radius $\Ap$ centered at
    each image source. The target $\xv\in B$ first touched by the nearest image disk is shown.
    For no aliasing at any targets in $B$, the image separation must thus be
    at least $\Ap+2$. The ratio in the diagram is shown accurately: $\Ap \approx 4.8$ in our
  tests.}
	\label{fig:imagelatt} 
\end{figure}

\subsection{Fourier quadrature truncation}
\label{sec:cut-off}

We truncate the infinite series in \eqref{eq:uhtrap} at a cut-off wavenumber magnitude
$K$ such that,  given an error tolerance $\epsilon$, the error is $\Oh{\epsilon}$.
Thus, the near history approximation takes the form
\beq\label{eq:uhtraptrunc}
\unh(\xv,t) \approx \left(\frac{\dk}{2\pi}\right)^2 \sum_{\nv\in\mathbb Z^2: \, \abs{\nv\dk}\leq K}
\!\!\!
\alpha(\nv\dk,t)\, e^{-i\nv\dk\cdot \xv},
\qquad \xv \in B. 
\eeq
To determine $K$,
we use the following {\it dimension-independent} theorem, proved in \cite{tkwfp3d}.
In short, it states that the $\alpha$ coefficients are small beyond a wavenumber magnitude
$K$ that is roughly the sum of the signature frequency cut-off $K_0$ and the
blending window frequency cut-off $2b/\delta$.
\begin{theorem}\label{thm:alphaDecay}
  Let $\sigma_j(t)\in L_2(\mathbb R_+)$,
  with $\|\sigma_j\|_1\le P$,
  $j=1,\dots,M$, be given source signature functions.
  Let $\epsilon$ denote the desired precision, 
with  $0<\epsilon < 1$, and let 
  the bandlimit $K_0$ be chosen so that
  the Fourier transforms $\hat{\sigma}_j(\omega)$ satisfy the decay estimate
\beq
\abs{\hat{\sigma}_j(\omega)} \leq \frac{\epsilon}{\omega^2},
\quad \text{for all }|\omega|>K_0.
\label{sighatdecay}
\eeq
Let the blending timescale be $\delta>0$,
and let $\phi$ be defined as in \eqref{eq:KBblending} 
with $b = \ln(1/\epsilon)$. Then, for each $\theta>1$,
the Fourier transform data defined by \eqref{eq:alphak} obey the decay condition
\beq
\abs{\alpha(\kv,t)} = \frac{C M \epsilon}{\kappa^3},
\qquad \text{for all }\kappa>K := K_0 + \frac{2b}{\delta\sqrt{1 - 1/\theta^2}},
\; t\in[0,T],
\label{albnd}
\eeq 
recalling that $\kappa:=|\kv|$.
Here $C$ is a constant independent of $\epsilon$,
that depends only weakly on $K_0$, $b$, $\delta$, $\theta$, and $P$.
\end{theorem}

\begin{remark}
  Its proof uses 
the decay estimate for $\hat{\phi'}$ in \eqref{eq:windowBandlimit}.
The estimate~\eqref{sighatdecay} is clearly satisfied
for $\sigma_j$ twice-differentiable, with even faster decay if
$\sigma_j$ is smoother.
\end{remark}

Now using that fact that $1/|\kv|^3$ is summable over the $\kv$ lattice in 2D,
and rolling in all constants, this theorem shows that the near history Fourier sum
truncation error satisfies
$$\sum_{\nv\in\mathbb Z^2:\abs{n\dk}>K}\alpha(\nv\dk,t)e^{-i\nv\dk\cdot\xv} = \Oh{\epsilon}.
$$
In practice we find it adequate to take the large-$\theta$ limit from the theorem,
and use simply set the cut-off to be the sum of the two bandwidths,
\beq\label{eq:K}
K = K_0 + \frac{2b}{\delta}.
\eeq

\subsection{The temporal blending timescale and the Duhamel update}
\label{s:delta}

With the choice of cut-off wavenumber $K$ in~\eqref{eq:K}, it remains to set the blending 
timescale $\delta$ and the time-step $\dt$.
We will set
\beq\label{eq:W}
\delta = W\dt,
\eeq
where $W$ is a small positive integer that can be used to balance the local and history costs.
For the Fourier coefficients $\alpha$ to be accurately resolved
in time requires that $\dt \lesssim \pi/K$, the Nyquist limit.
Then combining this with \eqref{eq:K} and \eqref{eq:W}
gives
\beq
\dt \lesssim \frac{\pi - 2b/W}{K_0},
\label{dt}
\eeq
which is necessarily somewhat less than the Nyquist limit $\pi/K_0$ sufficient to
resolve the signature functions $\sigma_j$ alone.
Increasing $W$ grows $\delta$ and thus the local cost, while reducing the near-history
cost.
For instance, for 8-digit accuracy ($\epsilon=10^{-8}$), 
$W = 24$ causes $K$ to be around twice the signal bandwidth $K_0$, hence
$\dt$ to be around half the Nyquist limit for the signature functions.
We typically choose $W$ in the range $10$-$30$.

The Fourier data $\alpha(\kv,t)$ defined by \eqref{eq:alphak},
while formally history-dependent, satisfy a simple 2-term recurrence relation
independently at each $\kv$.
\begin{lem}
  Let $\kv\in\R^2$.
  The exact evolution
  over one time step $\dt$ for the pair
  $\alpha(\kv,t)$ and $\dot \alpha(\kv,t) := \partial_t \alpha(\kv,t)$ is
\beq\label{eq:alphakEvolution}
\begin{split}
\alpha(\kv,t + \dt) &= \alpha(\kv,t)\cos(\kappa\dt) + \dot \alpha(\kv,t) \frac{\sin\kappa\dt}{\kappa} + h(\kv,t), \\
\dot\alpha(\kv,t + \dt) &= -\kappa\alpha(\kv,t)\sin(\kappa\dt) + \dot \alpha(\kv,t) \cos\kappa\dt+ g(\kv,t), \\
\end{split}
\eeq
where
\beq\label{eq:ghEquations}
\begin{split}
h(\kv,t) &:= \int_t^{t+\dt} \frac{\sin\kappa(t+\dt - \tau)}{\kappa} F(\kv,\tau)d\tau,
\;\; \\
g(\kv,t) &:= \int_t^{t + \dt}
\!\!\!\!\!
\cos\kappa(t + \dt - \tau) F(\kv,\tau)d\tau,
\end{split}
\eeq
and 
\beq\label{eq:alphakF}
F(\kv,t) = \int_{t - \delta}^{t}\left[ \Psi(\kv,t - \tau)\sighat(\kv,\tau) - \Psi_{\Ap}(\kv,t - \tau)\sighat(\kv,\tau - \Ap + \delta)\right]d\tau. 
\eeq
Here, $\Psi$ and $\Psi_{\Ap}$ are supported in $[0,\delta]$, and given by
\beq
\begin{split}
\Psi(\kv,\tau)&:= 2\cos\kappa\tau \, \phi'(\tau) + \frac{\sin\kappa\tau}{\kappa}\phi''(\tau), \\
\Psi_{\Ap}(\kv,\tau)&:= 2\cos\kappa(\tau + \Ap - \delta) \, \phi'(\tau) + \frac{\sin\kappa(\tau + \Ap-\delta)}{\kappa}\phi''(\tau).
\end{split}
\eeq
\end{lem}
\begin{proof}
It is straightforward to see that $\alpha(\kv,t)$ satisfies the ODE 
\beq
\begin{cases}
\ddot\alpha(\kv,t) + \kappa^2\alpha(\kv,t) = F(\kv,t), & t>0, \\
\alpha(\kv,0) = \dot\alpha(\kv,0) = 0, 
\end{cases}
\eeq
whose solution using the Duhamel principle is
\beq\label{eq:alphaDuhamel}
\alpha(\kv,t) = \int_0^t\frac{\sin\kappa(t - \tau)}{\kappa}F(\kv,\tau)d\tau,
\eeq
leading directly to~\eqref{eq:alphakEvolution}. 
\qed
\end{proof}
\begin{remark}
At the origin $\kv = \bold 0$, we take the limit $\kappa\rightarrow 0$ with $\sin\kappa \tau/\kappa\rightarrow \tau$. 
\end{remark}

Following the treatment in~\cite{tkwfp3d} and~\cite{wfp2025},
we use Gauss--Legendre quadrature over the interval $[t,t+\dt]$
to compute $h(\kv,t)$ and $g(\kv,t)$ in \eqref{eq:ghEquations}.
The function $F(\kv,t)$ in \eqref{eq:alphakF} is computed using the trapezoidal rule on the existing
time-stepping grid. Merging \eqref{eq:ghEquations} and \eqref{eq:alphakF}, and changing  
the order of integration leads to a formula of the form
\beq
\begin{aligned}
  h(\kv,t) &\approx \dt \sum_{m=0}^{W-1} p_m(\kv) \sighat(\kv,t-m\dt)
  - p^{(\Ap)}_m(\kv) \sighat(\kv,t-\Ap+\delta-m\dt),
  \\
  g(\kv,t) &\approx \dt \sum_{m=0}^{W-1} q_m(\kv) \sighat(\kv,t-m\dt)
  - q^{(\Ap)}_m(\kv) \sighat(\kv,t-\Ap+\delta-m\dt),
\end{aligned}
\label{hggrid}
\eeq
where $p_m$, $q_m$, $p^{(\Ap)}_m$, and $q^{(\Ap)}_m$ are independent of $t$ and 
can be precomputed for each $\kv$ in the Fourier quadrature grid.
We refer to~\cite{tkwfp3d} for further details.

\subsection{Computational complexity and storage}

At each time step, the evaluation of $\alpha(\kv,t)$ using~\eqref{eq:alphakEvolution}
requires the computation of $h(\kv,t)$ and $g(\kv,t)$ in \eqref{hggrid}. This requires
two calls to a (type I) 
non-uniform fast Fourier transform (NUFFT) to obtain $\sighat(\kv,t)$ and 
$\sighat(\kv,t - \Ap)$ at $W$ previous stages
\cite{finufft,finufftlib}.
This is done for each of the $N^2$ wave-vectors $\kv$ where $N = \ceil{2K/\dk}$. 
The total cost of these type I transforms is $\Oh{\log^2(1/\epsilon)M + N^2\log N}$,
recalling that $M$ is the number of sources.

Once all $\alpha(\kv,t)$ are known at a given time $t$, we apply a type II NUFFT 
to evaluate $\unh$ from~\eqref{eq:uhtraptrunc}, at a cost of
$\Oh{\log^2(1/\epsilon)\Nx + N^2\log N}$, where $\Nx$ is the number of target points. 
The algorithm requires access to the values $\sighat(\kv,t)$ and
$\sighat(\kv,t - \Ap+\delta)$ at $W$ time levels prior to $t - \Ap+\delta$ and $t$,
incurring a storage cost of $\Oh{NM}$ complex numbers, since
the near-history involves $\Ap/\dt = \Oh{K} = \Oh{N}$ time steps,
while $\sighat$ may be recomputed as needed via NUFFTs from the $M$ signatures $\sigma_j$
on the time grid.
\begin{remark}
  \label{r:dt0}
In the limit $\dt \rightarrow 0$, it follows from the preceding analysis
that $\delta \rightarrow 0$ as well, if $W$ is fixed in \eqref{eq:W}.
This would require that $K \rightarrow \infty$ and $N \rightarrow \infty$ 
to account for 
the sharp transition from the the local part to the near history in the frequency 
domain. One could 
avoid this by fixing $\delta =\Oh{1/K_0}$, which would force $W$ to grow,
putting an increased burden on the evaluation of the local part instead.
We have not made this modification to our algorithm.
In practice, since ours is a spectral scheme,
for efficiency one should always choose $\dt$ around $1/K_0$,
since the solution is temporally resolved on that time scale.
\end{remark}

\section{Evaluation of the far history}\label{sec:far}

The contribution of the far history is also computed in the Fourier transform domain;
recall \eqref{eq:ufh2} and \eqref{eq:alphaF}.
In order to efficiently evaluate the values $\alphaF(\kv,t)$,
we insert the sum-of-exponentials approximation \eqref{eq:LaplaceApprox}
into \eqref{eq:alphaF}.
Given $0 < \tilde\epsilon \ll 1$, and the $\Nlam$ quadrature nodes $\lambda_l$ and weights $q_l$, $l = 1, \dots, \Nlam$, after a little algebra, we get
\beq\label{eq:alphaFApprox}
\alphaF(\kv,t) =  \sum_{l = 1}^{N_\lambda}\calH_l(\kappa;A)\beta_l(\kv,t;\Ap) + \Oh{\tilde\epsilon} 
\eeq
where $\kappa = \abs{\kv}$, the radial Hankel transform coefficients are
\beq\label{eq:H}
\calH_l(\kappa;A) := q_l e^{-A\lambda_l}
\int_0^A J_0(\kappa r)I_0(\lambda_lr)
\phi_{\Delta}(A - r) rdr, 
\eeq
and all time-dependence is in the coefficients
\beq\label{eq:beta0}
\beta_l(\kv,t;\Ap) := e^{A\lambda_l}\int_{0}^{t - \Ap + \delta}e^{-\lambda_l(t - \tau)}\sighat(\kv,\tau)[1 - \phi(\tau - t + \Ap)]d\tau.
\eeq
The coefficients $\calH_l(\kappa;A)$ in \eqref{eq:H} do not depend on time, and 
involve smooth integrands over a bounded interval;
we precompute them using Gauss--Legendre quadrature
for each $l = 1, \dots, \Nlam$, at each $\kappa$.
(Note that we incorporate the factor
$e^{-A\lambda_l}$ in $\calH_l(\kappa;A)$ to compensate for
the exponential growth of $I_0(\lambda_lr)$ as $\lambda_l$ increases,
so that $\calH_l(\kappa;A) = \Oh{1}$ up to weak algebraic factors.)

For the time-dependent coefficients $\beta_l(\kv,t;\Ap)$, we apply the partition
\beq\label{eq:beta}
\beta_l(\kv,t;\Ap)=
e^{A\lambda_l} \left[\beta_l^{(1)}(\kv,t;\Ap) + \beta_l^{(2)}(\kv,t;\Ap)\right], 
\eeq
where $\beta_l^{(1)}$ covers the bulk of the far-history and $\beta_l^{(2)}$ just the transition region:
\beq
\begin{split}
 \beta_l^{(1)}(\kv,t;\Ap)&:=\int_{0}^{t - \Ap}e^{-\lambda_l(t - \tau)}\sighat(\kv,\tau)d\tau, \\
\text{and} \qquad   
\beta_l^{(2)}(\kv,t;\Ap)&:= \int_{t - \Ap}^{t - \Ap + \delta}e^{-\lambda_l(t - \tau)}\sighat(\kv,\tau)[1 - \phi(\tau - t + \Ap)]d\tau. 
\end{split}
\eeq
The factor $[1 - \phi(\tau - t + \Ap)] = 1$ for $\tau \in [0,t - \Ap]$, and is thus omitted
from the expression for $\beta_l^{(1)}(\kv,t;\Ap)$. 
Due to the exponential time kernel,
the coefficients $\beta_l^{(1)}$ can be computed at each time step using the
Duhamel recurrence
\beq\label{eq:betaEvolution}
\beta_l^{(1)}(\kv,t + \dt; \Ap) = e^{-\lambda_l\dt}\left[\beta_l^{(1)}(\kv,t;\Ap) + \int_{t - \Ap}^{t - \Ap + \dt}e^{-\lambda_l (t - \tau)}\sighat(\kv,\tau)d\tau\right] 
\eeq
for each $l = 1, \dots, \Nlam$. The integral in~\eqref{eq:betaEvolution} is smooth and
evaluated using Gauss--Legendre quadrature. 
The coefficients, $\beta_l^{(2)}$ in \eqref{eq:beta} 
involve only a few time steps (recalling that $\delta = W\dt$) and
can be computed directly using Gauss--Legendre quadrature for each~$l = 1, \dots, \Nlam$. 

Once the values
$\alphaF(\kv,t)$ are known, we approximate $\ufh$ from the Fourier integral \eqref{eq:ufh2},
using the trapezoidal rule with step size $\dk$, truncated at a maximum wavenumber $\Kf$:
\beq\label{eq:ufhApprox}
\ufh(\xv,t) \approx \left(\frac{\dk}{2\pi}\right)^2 \sum_{\nv\in\mathbb Z^{2}:\abs{\nv\dk}\leq\Kf}\alphaF(\nv\dk,t)e^{-i\nv\dk\cdot\xv}.
\eeq
It is convenient to use the same $\dk$ as the near-history in \eqref{eq:dk},
so that coefficients may be added, enabling a single Type II NUFFT to evaluate
the combined near and far history components.
For the reason discussed at the top of Sec.~\ref{sec:cut-off},
this quadrature rule is exact since the spatial support of $\calG_A$ is $r\le A < \Ap$.
Since the far history does not dominate (at least for high frequency problems),
there is little reason (and less convenience) to use the slightly larger $\dk$
in \eqref{eq:ufhApprox} allowed with zero aliasing error.

The cut-off frequency $\Kf$ is determined by the term $\calH_l(\kappa;A)$, 
which depends on $\phiD$ with $\Delta = \Oh{1}$. 
For $A = 2\sqrt{2} + \Delta$, we determine $\Kf$ experimentally for various values of $\Delta$, 
with the results listed in Table~\ref{tab:Kf_estimates}. 
We leave rigorous bounds on  the decay rate of $\alphaF(\kv,t)$ to future work.
In practice, we set $\Delta = 1$ and $\Kf = 80$, sufficient for near
double precision accuracy.

\begin{table}[t]
  \centering
\begin{tabular}{|c|c|c|c|}
\hline
$\Delta$ & $\Kf$  & $\max_{l}\abs{\calH_l(\kappa;A)}$ at $\abs{\kv} = \Kf$        \\
\hhline{|=|=|=|}
2        & 52              & $8.8704\times 10^{-16}$ \\ \hline
1        & 55                 & $9.4563\times10^{-13}$  \\ \hline
1        & 67              & $9.3524\times 10^{-15}$ \\ \hline
1        & 80                & $6.6176\times10^{-16}$  \\ \hline
0.5      & 134           & $8.3037\times 10^{-16}$ \\ \hline
\end{tabular}
\caption{Maximum sizes of the Hankel transform
  coefficients for far-history evaluation,
  at various cut-off wavenumbers $\Kf$ and radial blending width $\Delta$.
  The radial support is $A = 2\sqrt{2} + \Delta$.}
\label{tab:Kf_estimates}
\end{table}

\subsection{Computational complexity}\label{sec:ufh_complexity}

The precomputation of all $\calH_l(\kappa;A)$
requires $\Oh{\Nf \Nlam}$ work,
where $\Nf= (\Kf/\dk)^2$ is the total number of far-history Fourier grid points,
since almost all $\kappa$ values are distinct.
For the evaluation of $\beta_l(\kv,t;\Ap)$, 
we require values of $\sighat(\kv,t)$ at irregular points in time; these are 
interpolated from the stored values of $\sighat(\kv,t)$ on the uniform time grid.
A similar interpolation task arises in the calculation of the local part, and further 
details are provided in section~\ref{sec:local}). 
At each time step, computing $\beta_l(\kv,t;\Ap)$ in~\eqref{eq:beta} requires $\Oh{\Nlam \Nf}$
work, and so does evaluating $\alphaF(\kv,t)$.
We then add these coefficients to the relevant $\alpha(\kv,t)$ coefficients
so that the far history is incorporated into the type II NUFFT for the near history
evaluation at all targets.

\section{Evaluation of the local part}\label{sec:local}

Using the change of variables $\tau\mapsto t - \tau$
(so that $\tau$ now represents time delay into the past),
the local part \eqref{eq:ul1} is
\beq\label{eq:local2}
\ul(\xv,t) = \frac{1}{2\pi}\sum_{j\in\calNd(\xv)}\int_{r_j}^{\delta} \frac{\sigma_j(t - \tau)[1 - \phi(\tau)]}{\sqrt{\tau^2 - r_j^2}}d\tau, 
\eeq
where $r_j = |\xv - \yv_j|$, for $j = 1, \dots, M$,
recalling that $\calNd(\xv)$ represents the set of indices of sources
with distances $r_j \in (0,\delta)$ from the target $\xv$.
The integrand in \eqref{eq:local2} has an inverse square-root singularity
at $\tau = r_j$. However, when $r_j \ll 1$ the integrand
behaves like a pole in the domain $\tau \gg r_j$.
Thus we fix a transition point $r_0$ (that is in practice best set around $\dt/100$),
and use one scheme for $r>r_0$ and a different one for $r<r_0$.

When $r_j>r_0$, a single rule that handles the $-1/2$ power singularity is sufficient.
For this we change variable via $\tau = r_j + s^2$, so that each above integral becomes
\[
\int_0^{\sqrt{\delta - r_j}} \frac{2\sigma_j(t - r_j - s^2)[1 - \phi(r_j + s^2)]}{\sqrt{s^2 + 2r_j}}ds,
\]
which is smooth in $s$. Gauss--Legendre quadrature over this $s$ domain is then rapidly convergent,
needing around $\NL=60$ nodes for close to double precision accuracy.

For $r_j<r_0$ the singularity is followed by a large region with $\sim1/\tau$ behavior,
and for this the transformation $\tau = r_j \cosh s$ is much better because it
grows exponentially with $s$, yet still handles the $-1/2$ power at $\tau=r_j$.
However, for $\tau$ of order $\dt$ up to $\delta$, the nodes for that transformation would
be too coarse. Thus it is more efficient to split this at $\tau_0=2\dt$,
giving upper interval handled directly, and making the integral into
\[
\int_0^{\cosh^{-1}(\tau_0/r_j)}\sigma_j(t - r_j\cosh s)[1 - \phi(r_j\cosh s)]ds \; +\;
\int_{\tau_0}^{\delta} \frac{\sigma_j(t - \tau)[1 - \phi(\tau)]}{\sqrt{\tau^2 - r_j^2}}d\tau
.
\]
We then use Gauss--Legendre quadrature in $s$ for the first interval, and
in $\tau$ for the second. Experiments show that apportioning roughly half of the
nodes to each interval is efficient down to $r_j = 10^{-5}$.
Around $\NL=80$ total nodes are needed to get 12 accurate digits.

With the numerical integration procedure above in hand, we may write
\beq\label{eq:ul_GL}
\ul(\xv_i,t) \approx \frac{1}{2\pi} \sum_{j\in\calNd(\xv_i)}\sum_{m = 1}^{\NL} w_{m}^{(ij)} \sigma_j(t - s_{m}^{(ij)}),
\qquad i = 1, \dots, \Nx,
\eeq
where the quadrature nodes $s_m^{(ij)}$,
$m = 1, \dots, \NL$,
and weights $w^{(ij)}_m$ (which include all factors except $\sigma_j$),
depend only on  the source-target distance $r_{ij} = |\xv_i-\yv_j|$.

The above demands signature values
at times that do not lie on the uniform grid $t_n = n\dt$, $n = 1, \dots, \Nt$, with $\Nt\dt = T$.
Thus we approximate $\sigma_j(t - s_{m}^{(ij)})$ using $p$th-order local interpolation from 
the $p$ nearest values of $\sigma_j$ on the uniform time grid. 
Suppose now that we save $\nm$ time levels of $\sigma_j$ prior to the current time step $t_n$.
Then, for each target $\xv_i$, source $\yv_j$, and corresponding interpolation node $s_{m}^{(ij)}$, we 
compute interpolation weights $\xi_{m,l}^{(ij)}$, where $l = 1, \dots,\nm$ such that 
\beq\label{eq:sigInterp}
\sigma_j(t_n - s_{m}^{(ij)}) \approx \sum_{l = 1}^{\nm}\xi_{m,l}^{(ij)} \sigma_j\left(t_{n - \nm+l}\right). 
\eeq
For our scheme, $\nm = W + 1 + \lceil p/2\rceil$, since the near-history evaluation requires 
$W$ steps in the past, and the $p$th order interpolation requires $\lceil p/2\rceil$ extra prior 
time levels. Such past levels are available even at the first time step since $\sigma_j(t\leq 0)$ 
is assumed to be zero for all $j = 1, \dots, M$.
The approximation of $\ul$ can then be written as
\beq
\ul(\xv_i,t_n) \approx \frac{1}{2\pi} \sum_{j \in\calNd(\xv_i)}\sum_{l = 1}^{\nm} \eta_l^{(ij)}\sigma_j\left(t_{n - \nm+l}\right), \quad\text{ where }\quad\eta_l^{(ij)} = \sum_{m = 1}^{\NL} w_m^{(ij)} \xi_{m,l}^{(ij)}. 
\eeq 
In short, omitting the somewhat tedious algebra, it is straightforward to 
evaluate $\ul$ at the full set of target points using a sparse matrix-vector product.
Assuming the sources are uniformly distributed in the computational domain, and that $\dt$
is equal to their average spacing, the amount of local work is of the order
$\Oh{ \Nx W^3}$.
Here a factor $W^2$ estimates the number of sources in the near field of each target,
while $\nm \approx W+p/2 = \Oh{W}$.
The storage needed for the local evaluation (and the corresponding sparse matrix) 
is of the same order.

\section{Numerical results}\label{sec:results}

We now illustrate the performance and accuracy of the 2D TK-WFP method
when evaluating solutions of the free-space wave equation with given
source signature functions.
We fix the parameters $\Delta = 1$ and $a = 1$ in \eqref{A}--\eqref{Ap}.
The code was implemented in MATLAB (version R2023b),
without using any explicit parallelization.
This calls the parallel C++ library FINUFFT (version 2.4.1)
\cite{finufft,finufftlib} for all NUFFTs, where we set {\tt opts.nthreads=32}
(since larger thread numbers were counter-productive).

In our first example~(\ref{sec:conv}), 
we study the convergence of the method with $\dt$, for various
interpolation orders $p$ discussed in section~\ref{sec:local}.
We then consider three large-scale problems using a regular $900\times 900$ 
target grid covering the computational domain
$B = [-1,1]^2$,
with time-signature functions
containing frequencies ranging from $0$ to $300\pi$, corresponding
to 300 wavelengths across each side of $B$.
In example~(\ref{sec:random}), we place $10^6$ sources at random locations 
in $B$.
In example~(\ref{sec:circle}), we place
$10^5$ sources on a circle with increasing frequency content as the source
sweeps counterclockwise in angle.
In example~(\ref{sec:wiggly}), sources with random frequency content are 
located on a complicated closed curve contained in~$B$.
These last two curve examples model the application to time-domain boundary integral equations.

We use $\tilde u$ to denote the approximation to \eqref{eq:solnRep}
resulting from the 2D TK-WFP algorithm.
Its error is estimated by evaluating the exact solution in~\eqref{eq:solnRep} 
using Gauss--Legendre quadrature applied to the formula
\beq\label{eq:exactSol}
\begin{split}
u(\xv,t) & =  \frac{1}{\pi} \sum_{\stackrel{j = 1}{r_j<t, \, r_j\neq 0}}^M \int_0^{\sqrt{t - r_j}}\frac{\sigma_j(t - r_j - s^2)}{\sqrt{s^2 + 2r_j}}ds,\qquad r_j = \abs{\xv - \yv_j},
\end{split}
\eeq
making use of the change of variable $\tau = s^2 + r_j$ 
to handle the square-root singularities.
Since direct evaluation is prohibitively expensive, we evaluate the 
error on a subset of the full target grid 
in $B$, and only at the final time $T$. 
For a given $\dt$,
we define the absolute and relative error in the max norm over targets, by
\beq
\calE_{\dt}(t) := \norm{u(\cdot,t) - \tilde u(\cdot,t)}_\infty , 
\qquad \tilde\calE_{\dt}(t) := \frac{\calE_{\dt}(t)}{\norm{u(\cdot,t)}_\infty},
\eeq
respectively. Unless otherwise indicated, we use the final time $t = T$.

In all our numerical examples, we choose time signatures of the form
\beq\label{eq:exampleDensities}
\sigma_j(t) = 0.5\left[\erf(5(t - t_{0,j})) + 1\right]\sin(\omega_{j}(t-t_{0,j})), \qquad j = 1, \dots, M,
\eeq
where $t_{0,j}$ is a time offset and $\omega_{j}$ an oscillation frequency, different
for each source. The result is a truly wideband wave field.
Here the erf insures a smooth ``switch-on'' while
slightly growing the bandwidth beyond $\omega_j$.
A good approximation to the frequency beyond which all Fourier transforms are $\Oh{\epsilon}$
(i.e., the $\epsilon$-bandwidth)
is $K_0 = \max_j |\omega_j| + 10 \sqrt{\log(1/\epsilon)}$, the second
term corresponding to the additional frequency content from the erf.

For the large-scale high-frequency
Examples~\ref{sec:random},~\ref{sec:circle}, and~\ref{sec:wiggly}, 
we will set the error tolerance $\epsilon = 10^{-7}$,
interpolation order $p=20$, and the rather small $W = 16$ time steps for the WFP blending width.
Given this, one has $K \approx 2.8 K_0$,
and we set $\dt$ according to \eqref{dt}.
These choices were determined experimentally to yield around 6 digits of relative accuracy.

\subsection{Convergence rate}\label{sec:conv}

In this small-scale test we
place $M=100$ sources randomly in
$B$, with time signatures of the 
form~\eqref{eq:exampleDensities} with $t_{0,j}$ randomly
assigned in $[1.5,7]$, and random frequencies $\omega_j\in[0,10\pi]$.
For this example only we set $\epsilon = 10^{-8}$ and fix $W = 24$.
We then compute the error on a regular $10\times 10$ target grid at the
final time $T = 8$ as a function of the time step $\dt$, for various 
interpolation orders $p = 2, 4, \dots, 10$.
Recall that, in our implementation $K$ must grow
according to \eqref{eq:K} as $\dt\to0$, since $\delta = W\dt$ and $W$ is fixed;
also see Remark~\ref{r:dt0}.
In Figure~\ref{fig:exp5}, we plot the relative error versus $\dt$, for each 
value of $p$. The plot indicates that the rates of convergence match the design
order of accuracy, plateauing at around 1 digit worse than the specified tolerance.

\begin{figure}[ht]
	\centering
		\includegraphics[width=0.65\textwidth]{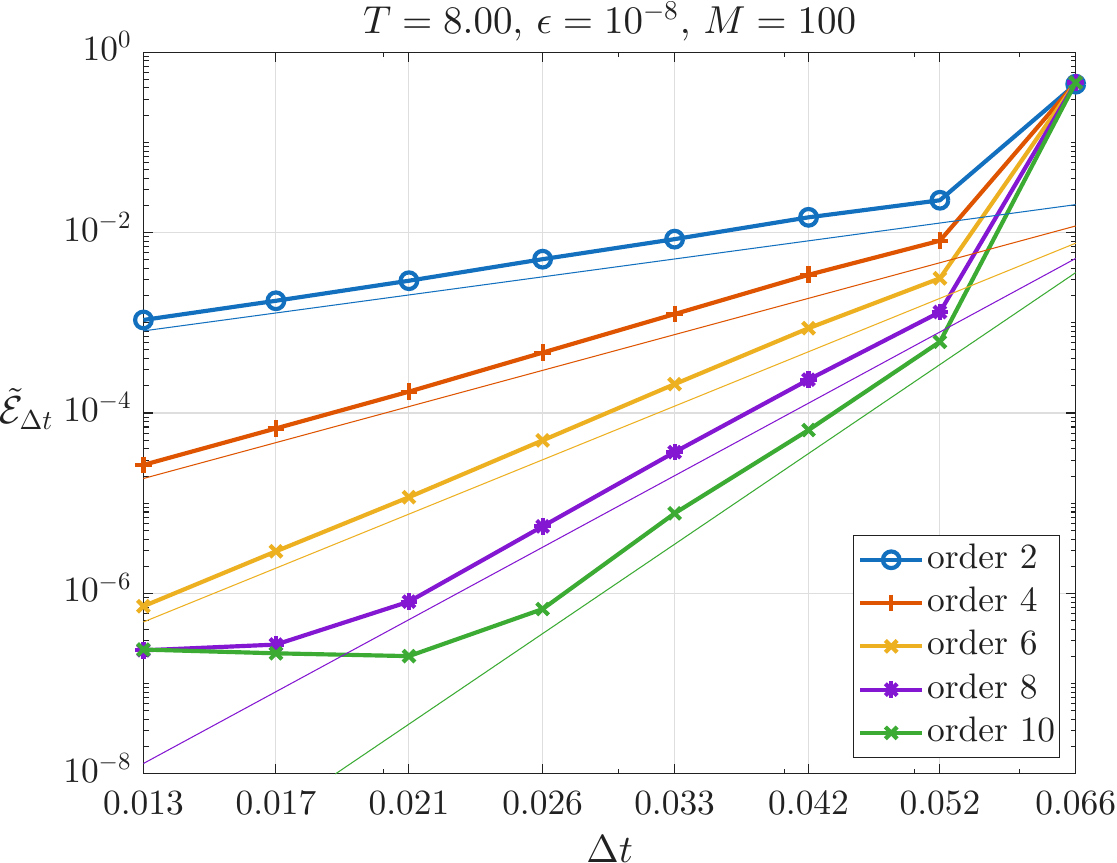}
	\caption{Convergence of the final-time solution computed with the 2D TK-WFP 
          algorithm with various interpolation orders $p$. The thin reference lines with matching colors have slopes $p$.
        }
	\label{fig:exp5} 
\end{figure}

\subsection{Random space-filling sources}\label{sec:random}

To study the performance of the algorithm on a large scale problem,
we set the number of sources to be $M = 10^6$ with random locations 
$\yv_j\in B$, $j = 1, \dots, M$, each having a time signature of the 
form~\eqref{eq:exampleDensities} with random offsets $t_{0,j}\in[1.5,7]$ and 
frequencies $\omega_j = 300\pi z_j^{1/3}$ where $z_j$ are independent and identically distributed (i.i.d.) uniform random samples from $[0,1]$.
Here the $1/3$ power serves to boost the high frequency content, while still covering
the full range $[0,300\pi]$.
We set the final time $T = 8$ and use a time step $\dt = 0.00112$, and set the
cutoff wavenumber $K = 2808$. With $\dk = 0.918$, 
the total number of Fourier modes for the near history evaluation is $N^2 = 6121^2$.
The Fourier modes for the far history are a subset of the near history $\kv$-grid,
and total $\Nf^2 = 125^2$ modes (each with $\Nlam=640$ terms).
The total number of time steps needed is $\Nt = 7150$. 
For this example, $K_0 = 983$, at which frequency each side of the computational domain $B$ 
is $K_0/\pi \approx 313$ wavelengths.
We compute the solution $u$ on a $\Nx = 900\times900$ regular mesh of target 
points. With $\delta  = 0.0179$, the typical number of sources 
within a $\delta$-neighborhood of a target point is $250$.

Figure~\ref{fig:exp1} shows the computed solution 
at $t = 4$ and $t = 8$, with a $5\times$ zoomed-in view in the 
subdomain $[0.4,0.6]^2$. For the zoomed-in window,
we calculate the solution over a fine $360\times360$ grid.  
The relative error, computed at the final time $T = 8$ on a $5\times5$ 
subset of the target mesh, is $\tilde\calE_{\dt} =   5.31\times 10^{-7}$.

Table~\ref{tab:exp1_timings} shows the time it takes to evaluate each 
solution component in \eqref{eq:split}, using
an AMD Rome node with two 64-core EPYC 7742 2.8 GHz CPUs and 1024 GB RAM, 
with 422 GB of the memory utilized. 
We evaluate the solution at 8 time slices only, for $t = 1, 2,\dots, 8$.
If the solution were evaluated using 2D TK-WFP at each of the 
$\Nt$ time steps, the estimated CPU time is 87 hours.
The naive direct evaluation of $u$ using \eqref{eq:exactSol} needs
3600 Gauss--Legendre nodes (which we verified were required, because of the high frequency
content).
Our single-threaded direct evaluation at all time steps
is estimated to demand about $6\times10^8$ hours, so that
the speed-up factor is $6\times 10^{6}$.
Since our 2D TK-WFP implementation exploits up to 32 parallel threads (in the NUFFTs),
it has some average parallel acceleration factor significantly less than 32.
Yet, compensating for this parallel factor still leaves the 2D TK-WFP
algorithm {\it at least $10^5$ times more efficient than direct evaluation}.

\begin{figure}[t]
	\includegraphics[width=0.49\textwidth]{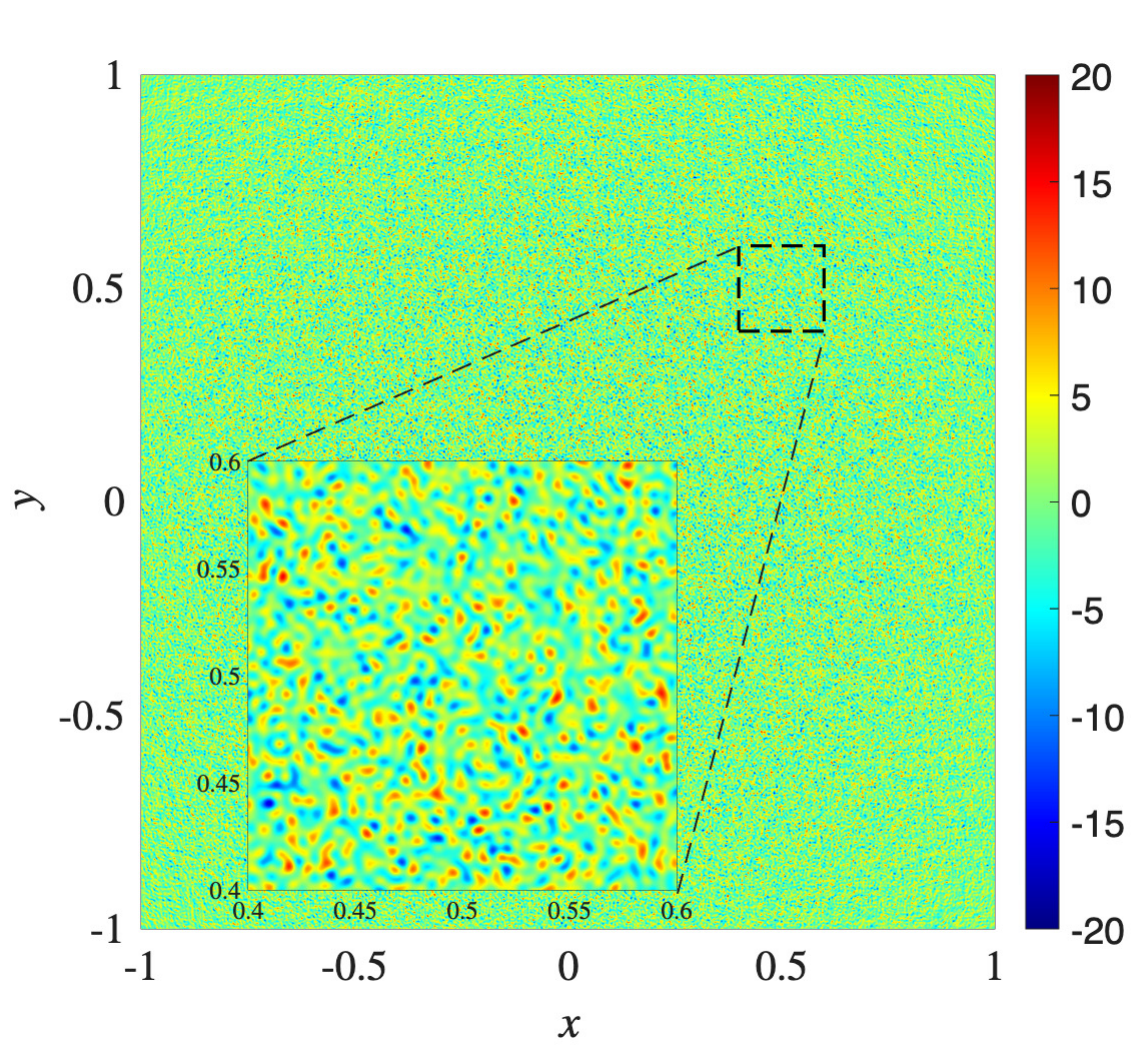}
	\includegraphics[width=0.49\textwidth]{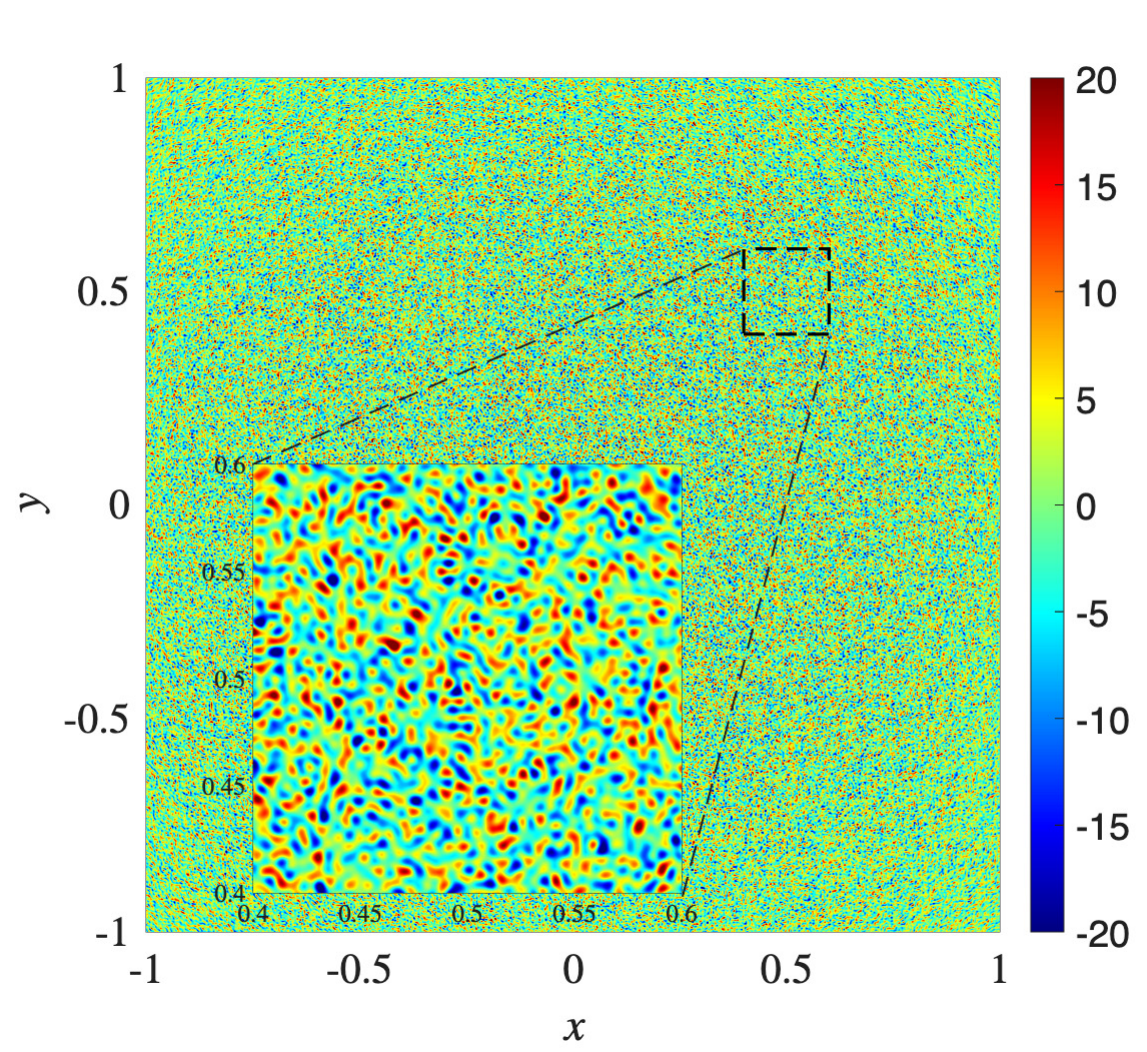}
	\caption{Computed solution $\tilde u$ at $t = 4$ (left) and $t = 8$ (right) for example \ref{sec:random}, on a $900\times 900$ target mesh, with $M = 10^6$ sources.
          $5\times$ zoomed-in views are inset.
          The solution contains all frequencies from zero up to 
          a maximum frequency, at which the domain is $313$ wavelengths on a side.
          The estimated relative maximum error 
at $t = 8$ is $\tilde\calE =   5.3\times 10^{-7}$.}
	\label{fig:exp1}
\end{figure}

\begin{table}[ht]
	\centering
	\begin{tabular}[t]{|l|l|}
	\hline
  Task & CPU time \\
  \hhline{|=|=|}
  precomputation                         & 7.6 h   \\ \hline
  $\ul$ eval.\ per time-step              & 23.3 sec \\
  $\uh$ eval.\ per time-step              & \ 0.6 sec \\
  $\ufh$ eval.\ per time-step              & \ 0.1 sec \\
  $\alpha(\kv,t)$ update per time-step   & \ 9.1 sec \\ 
  $\beta(\kv,t)$ update per time-step   & \ 0.02 sec \\ 
  Type I NUFFT update per time-step & \ 6.7 sec\\ \hline
  Total per time-step (all targets)                    &  39.8 sec \\ \hline
  Direct $u$ eval.\ per time-step {\it per target}       & \ 5.9 min   \\ \hhline{|=|=|}
  2D TK-WFP, total for $0\le t\le 8$         &  86.6 h (est.) \\ \hline
  Direct eval, total for $0\le t\le 8$   & $ 5.7\times10^{8}$ h (est.)  \\ \hline
\end{tabular}
	\caption{Breakdown of CPU timings for Example~\ref{sec:random}
          with $M=1000000$ sources and $\Nx=810000$ targets.
          The total 2D TK-WFP cost over all time-steps, and the total naive direct quadrature
          evaluation cost
          are estimated from their average costs for one time-step,
indicated by~(est.). }
	\label{tab:exp1_timings} 
\end{table}

\subsection{Sources on a circle}\label{sec:circle}

To illustrate a discretized space-time layer potential,
we place $M = 10^5$ equispaced sources on the circle
\beq
\xv(s) = [0.8\cos s + 0.2, 0.8\sin s + 0.2],\qquad s\in[0,2\pi],
\eeq
which touches the boundary of $B$.
The signatures are as in \eqref{eq:exampleDensities}
with $t_{0,j}\in[1.5,7]$ and $\omega_j\in[0,300\pi]$
both linearly increasing as $s$ varies from $0$ to $2\pi$:
the waves sweep around the circle while growing in frequency.
With a final time $T = 8$, the parameters $\dt$, $\Nt$, $K_0$, $K$, $N$, $\dk$, 
and $\Nf$, are set as in the previous example.
Again we evaluate the solution on a $900\times900$ target grid.
Each target point has on average $25$ sources in its $\delta$-neighborhood
(although most targets have none).

Figure~\ref{fig:exp2} shows snapshots of the computed solution,
along with $10\times$ zoomed-in views near caustic areas as shown by the boxes.
Interacting wave fronts of increasing frequency are clearly visible,
as is the high resolution obtained by the method. 
The relative error at the final time is estimated on a $5\times 5$ subset of 
the target grid, giving $\tilde{\calE} = 7.9\times 10^{-6}$.

Table~\ref{tab:exp2_timings} shows the time required for each 
solution component in \eqref{eq:split}, using
an AMD Rome node with two 64-core EPYC 7742 3.4 GHz CPUs and 1024 GB RAM, 
with 340 GB of the memory utilized.
Again allowing for a parallel acceleration factor of $\approx30$,
our algorithm is still around $10^{5}$ times more efficient 
than direct evaluation.

\begin{figure}[t]
          \includegraphics[width=0.49\textwidth]{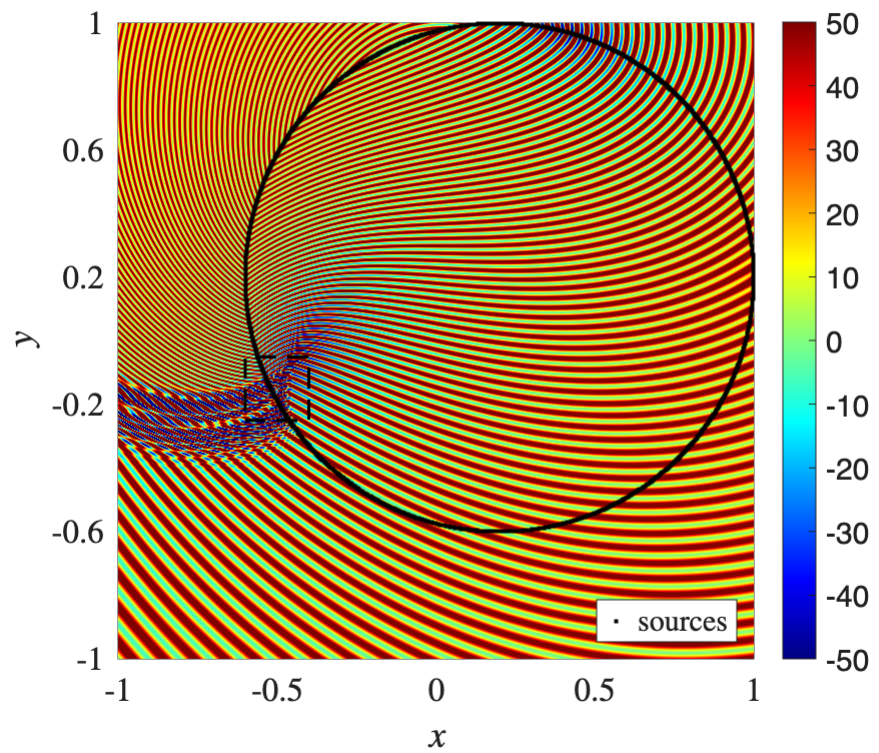}
          \includegraphics[width=0.49\textwidth]{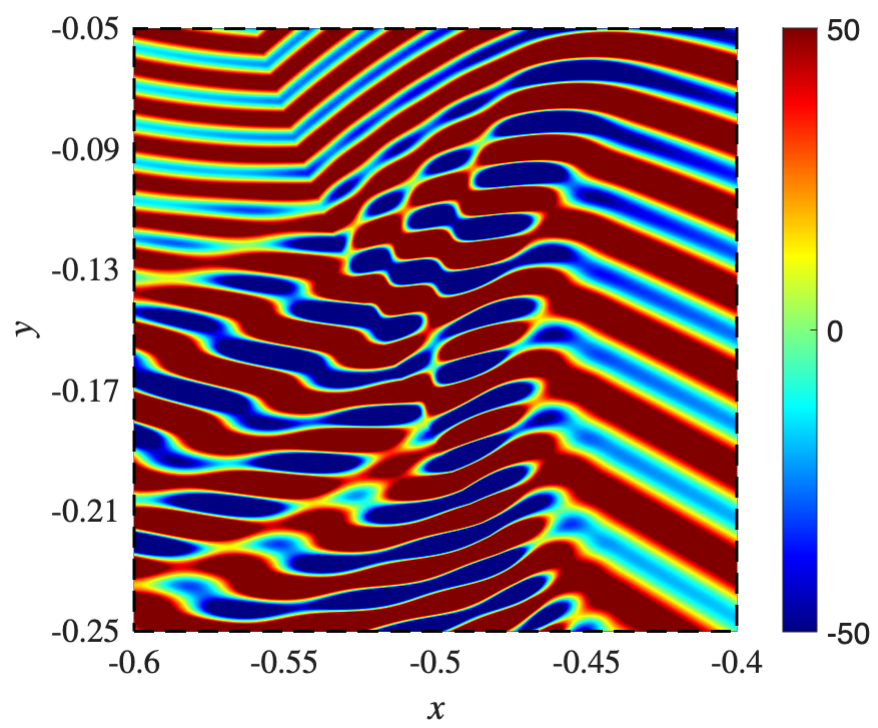}
          \\
          \includegraphics[width=0.49\textwidth]{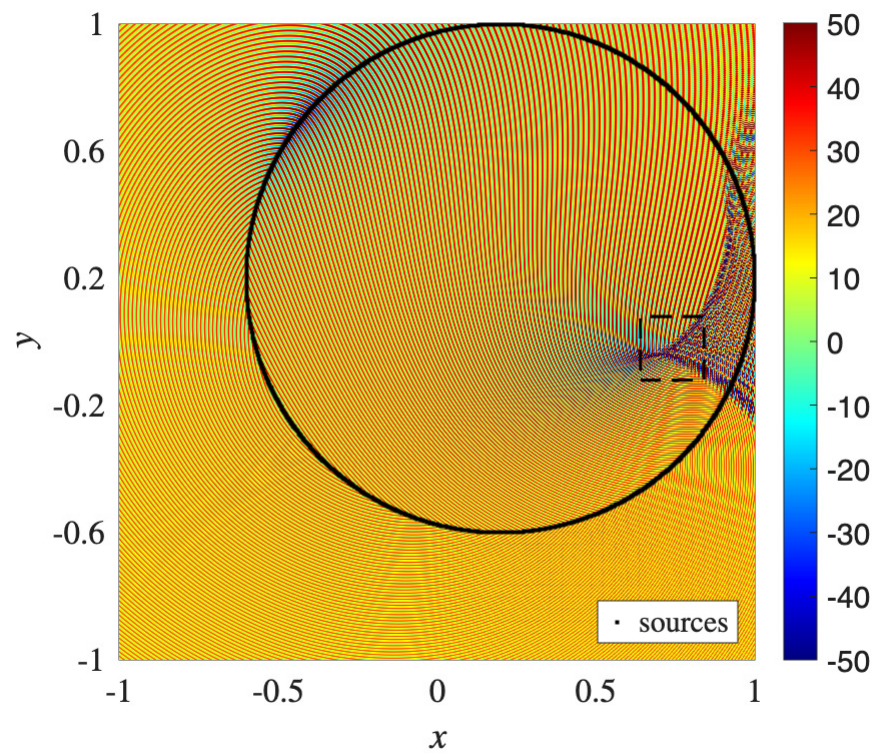}
          \includegraphics[width=0.49\textwidth]{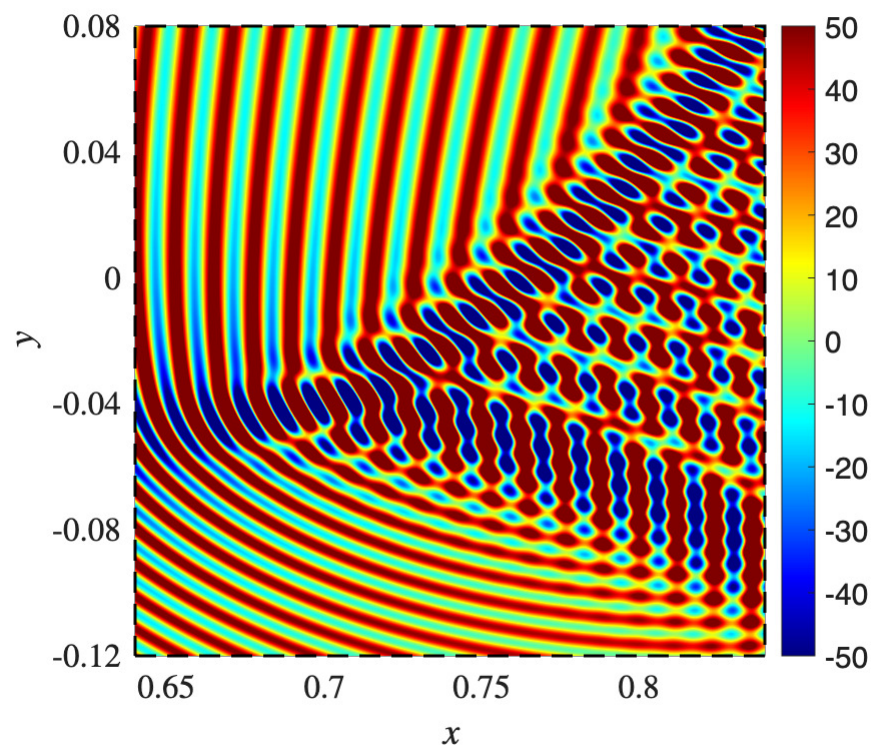}
          \vspace{-2ex}
\caption{Computed solution for example~\ref{sec:circle} at $t = 5$ and $t = 8$ (top left and bottom left) with $10^5$ sources on a circle, with deterministic
frequencies increasing up to $300\pi$ as the circle is swept in a counter-clockwise
direction. To the right of each plot is a $10\times$ zoomed-in view,
to subregion $[-0.6,-0.4]\times[-0.25,-0.05]$ (top) or $[0.64,0.84]\times[-0.12,0.08]$
(bottom), shown as dashed squares in the left plots.
The relative error computed 
at the final time is $\tilde\calE = 7.88\times 10^{-6}$. }
		\label{fig:exp2} 
	\end{figure}

  \begin{table}[ht]
	\centering
	\begin{tabular}[t]{|l|l|}
	\hline
  Task & CPU time \\
  \hhline{|=|=|}
  precomputation                         & 52 min   \\ \hline
  $\ul$ eval.\ per time-step              & 1.1 sec \\
  $\uh$ eval.\ per time-step              & 0.6 sec \\
  $\ufh$ eval.\ per time-step              & 0.1 sec \\
  $\alpha(\kv,t)$ update per time-step   &  8.9 sec \\ 
  $\beta(\kv,t)$ update per time-step   &  0.02 sec \\ 
  Type I NUFFT update per time-step & 5.3 sec\\ \hline
  Total per time-step (all targets)                   &  16 sec \\ \hline
  Direct $u$ eval.\ per time-step {\it per target}       & 56 sec   \\ \hhline{|=|=|}
  2D TK-WFP, total for $0\le t\le 8$         &  33 h (est.) \\ \hline
  Direct eval, total for $0\le t\le 8$   & $ 9 \times10^{7}$ h (est.)  \\ \hline
\end{tabular}
	\caption{Breakdown of CPU timings for Example~\ref{sec:circle},
          with $100000$ sources on a circle, and $810000$ targets.
        }
	\label{tab:exp2_timings} 
\end{table}

  \begin{figure}[ht]
  \includegraphics[width=0.47\textwidth]{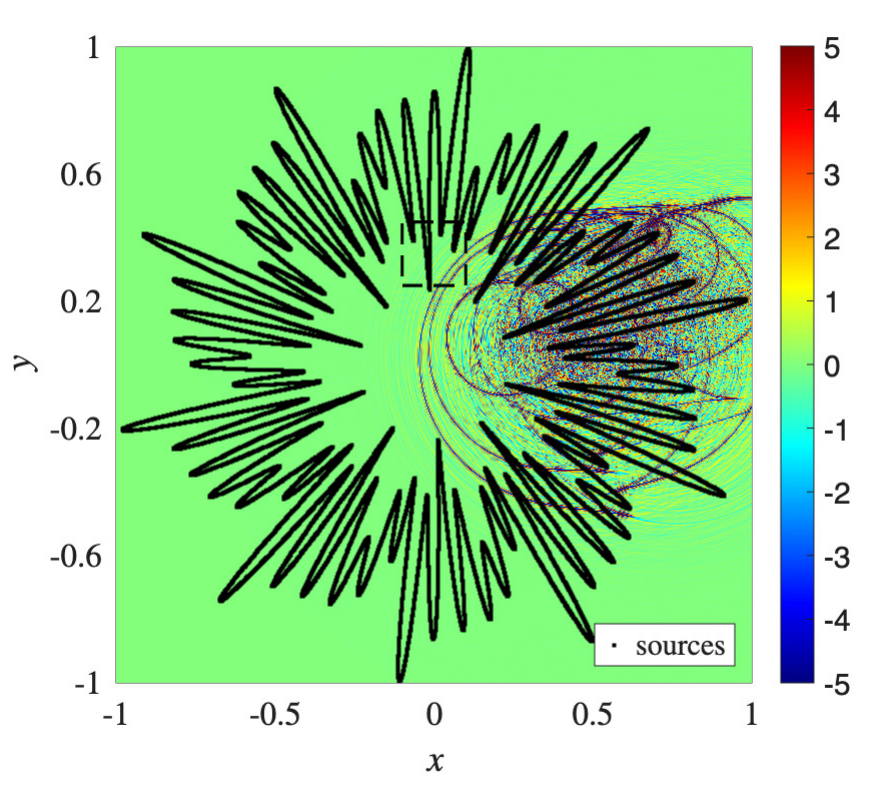}
  \includegraphics[width=0.47\textwidth]{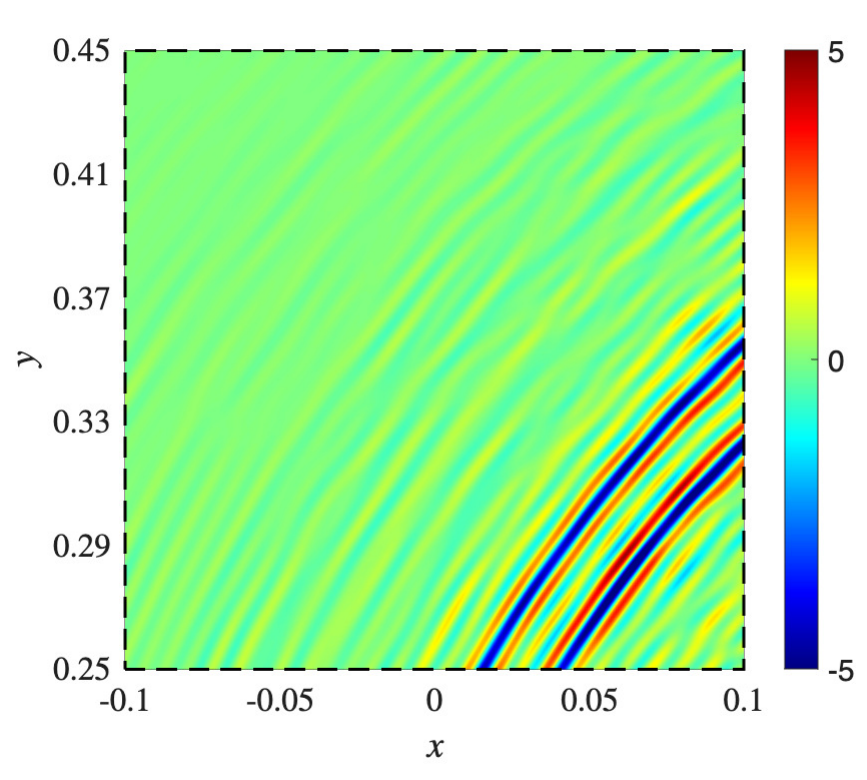}
  \\
  \includegraphics[width=0.47\textwidth]{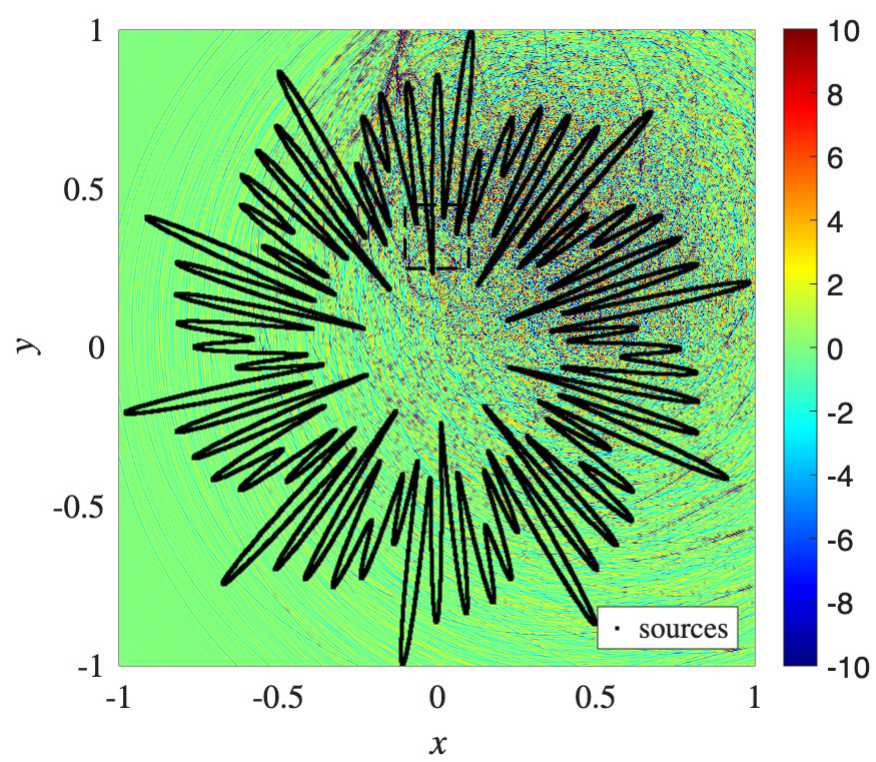}
  \includegraphics[width=0.47\textwidth]{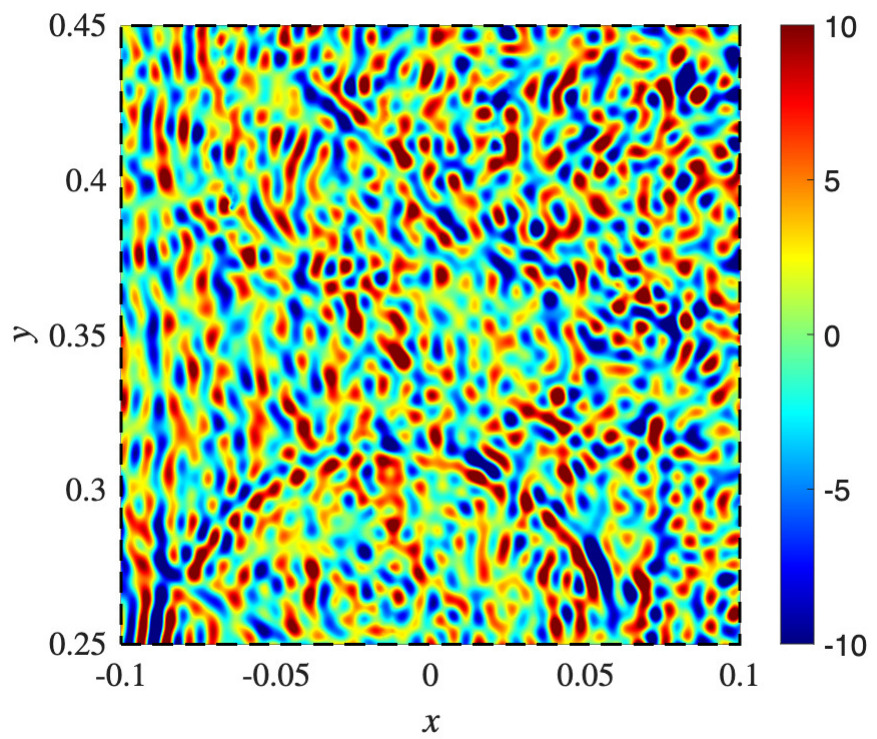}
  \\
  \includegraphics[width=0.47\textwidth]{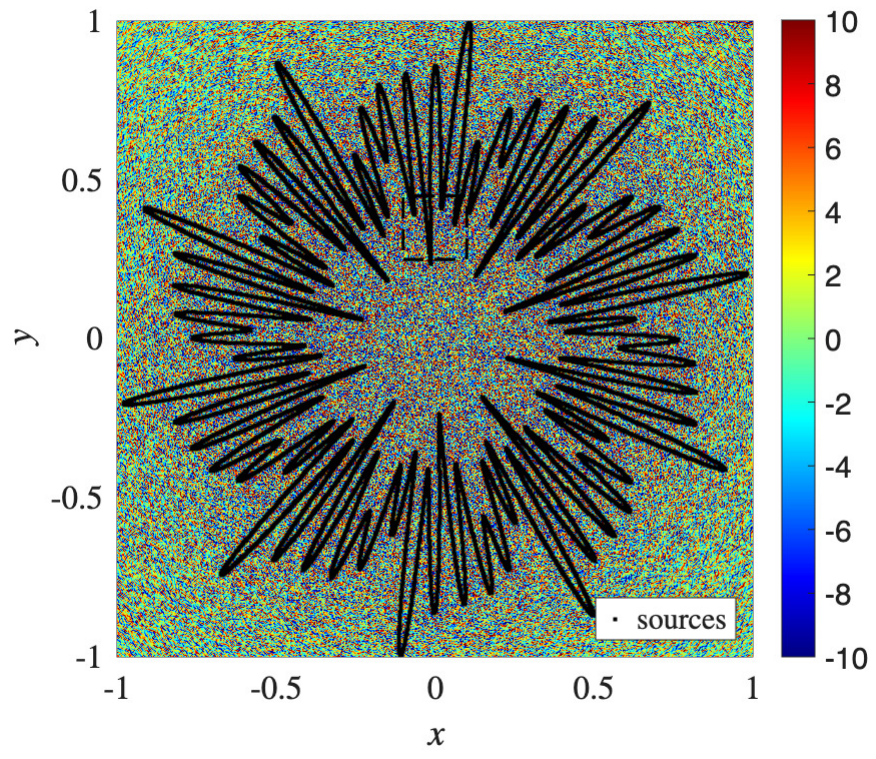}
  \includegraphics[width=0.47\textwidth]{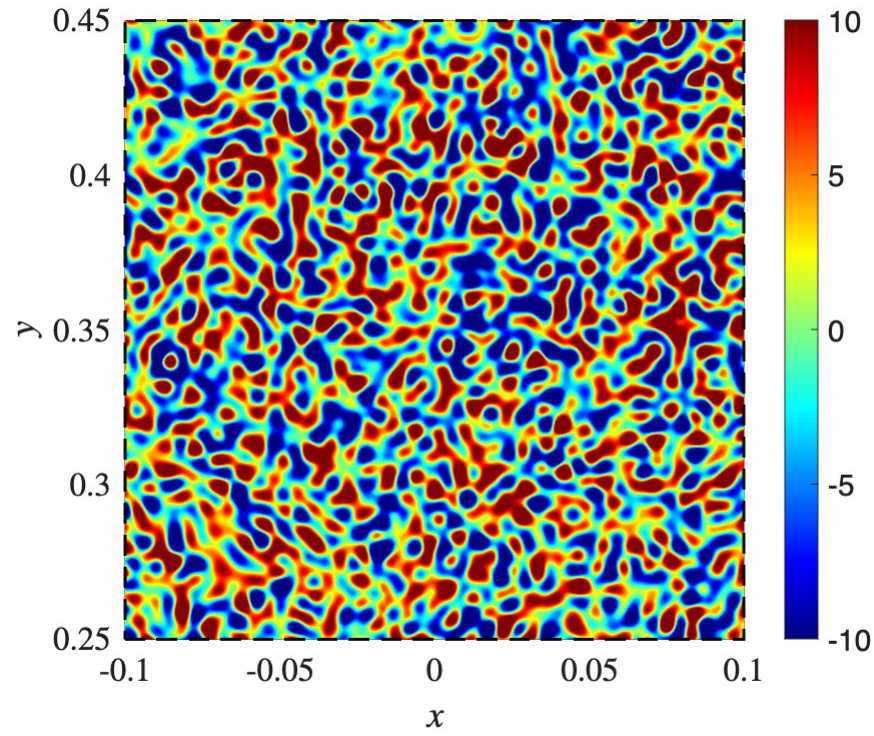}
  \vspace{-2ex}                
  \caption{Computed solution $\tilde u$ for Example~\ref{sec:wiggly},
    at $t = 2,3$, and $8$ (top to bottom) and $10\times$ zoomed-in views (right column)
    in the fixed subdomain $[-0.1,0.1]\times[0.25,0.45]$.
    There are $M=10^6$ sources on the curve, and a $900\times900$ target grid.
    The side of the square domain $[-1,1]^2$ is about $313$ wavelengths,
    at the maximum frequency of the source time signatures.
    The relative error in the max norm is estimated to be
    $\tilde{\calE} = 4.3\times10^{-7}$.}
  \label{fig:exp9} 
\end{figure}

\begin{table}[th]
\centering
	\begin{tabular}[t]{|l|l|}
	\hline
  Task & CPU time \\
  \hhline{|=|=|}
  precomputation                         & 8 h   \\ \hline
  $\ul$ eval.\ per time-step              & 26 sec \\
  $\uh$ eval.\ per time-step              & \ 0.6 sec \\
  $\ufh$ eval.\ per time-step              & \ 0.1 sec \\
  $\alpha(\kv,t)$ update per time-step   & \ 9 sec \\ 
  $\beta(\kv,t)$ update per time-step   & \ 0.02 sec \\ 
  Type I NUFFT update per time-step & \ 6.6 sec\\ \hline
  Total per time-step (all targets)                    &  42.4 sec \\ \hline
  Direct $u$ eval.\ per time-step {\it per target}       & 7.5 min   \\ \hhline{|=|=|}
  2D TK-WFP, total for $0\le t\le 8$         &  92 h (est.) \\ \hline
  Direct eval, total for $0\le t\le 8$   &  $7.2\times10^{8}$ h (est.)  \\ \hline
\end{tabular}
\caption{CPU timings for Example~\ref{sec:wiggly} on the complicated
  curve,
  with $1000000$ sources and $810000$ targets,
  comparing
  the cost of the proposed method to direct evaluation.}
\label{tab:exp9_timings}
\end{table}

\subsection{Sources on a complicated closed curve}\label{sec:wiggly}

In our final experiment, we place
$M = 10^6$ sources on on a highly-oscillatory curve
\beq
\begin{gathered}
\xv(s) = [r(s)\cos s, r(s)\sin s],\qquad s\in[0,2\pi], \\
 r(s) =  0.61 + 0.2\cos(60s) - 0.1\sin(20s) + 0.05\cos(30s)- 0.1\cos(40s). 
\end{gathered}
\eeq
Each source has a time signature $\sigma_j$ defined 
as in \eqref{eq:exampleDensities}
with time offsets $t_{0,j}\in[1.5,7]$ assigned in increasing order as $s$ increases,
but $\omega_j = 300\pi z_j^{1/3}$ where $z_j$ are uniformly random i.i.d.\ samples
on $[0,1]$, as before.
We set the final time $T = 8$ and used
the same parameters $\dt$, $\Nt$, $K_0$, $K$, $N$, $\dk$, and $\Nf$ 
as in the previous two examples. 
The solution is computed on a $900\times900$ target grid. Each target point 
has, on average, $250$ sources in its $\delta$-neighborhood. 

Figure~\ref{fig:exp9} shows the computed solution at $t = 2, \ 3$, and $8$ along with
$10\times$ zoomed-in views that renders the wavelength visible.
Some rather curious patterns are generated in these examples. At time 
$t=2$, for example, circular fronts seem to emanate from regions of high curvature along the
curve. We do not attempt an explanation of this here, noting only that
these calculations are fully resolved;
the relative error at $T = 8$ on a $5\times 5$ subset of the target grid 
is $\tilde{\calE} = 4.3\times10^{-7}$. 

Table~\ref{tab:exp9_timings} shows the time required (as in the 
previous examples), using
an AMD Rome node with two 64-core EPYC 7742 3.4 GHz CPUs and 1024 GB RAM, 
with 442 GB of the memory utilized. 
Direct evaluation of $u$ from \eqref{eq:exactSol} used again a 3600-point
Gauss--Legendre quadrature.
Allowing for parallelism as before, the
estimated speed-up factor over direct evaluation exceeds $2\times10^5$.

\section{Conclusion}\label{sec:conclusion}

We have introduced the 2D 
truncated-kernel, windowed Fourier projection (2D TK-WFP) algorithm for 
evaluating free-space hyperbolic potentials
due to wideband point sources,
with sources and targets confined to a bounded domain. 
The algorithm compresses the ``history part'' of the solution 
in \eqref{eq:solnRep} by splitting it into a non-smooth local part,
evaluated using direct quadrature, and two smooth components, the 
near history and the far history. Both
are approximated by Fourier representations, using 
the non-uniform fast Fourier transform
\cite{finufft,finufftlib}. 
Such a partition is made efficient through the use of narrow temporal,
and wide radial, smooth blending functions.
No radiation boundary condition is required.

A critical component of the 2D TK-WFP method is the suppression of 
high frequency oscillations in the spectral representation of the far history,
associated with the weak Huygens' principle and hence absent in the 3D case.
This is accomplished by using a radially-truncated variant of the spectral
wave kernel, combined with a sum-of-exponential approximation that allows us to
develop an efficient recurrence in time for each spectral mode (exploiting
the semigroup structure).
The total cost of the algorithm is quasi-linear with respect to the number of sources $M$ and time steps $\Nt$, while direct evaluation is quadratic in each.
The total number $N^2$ of Fourier modes required, however, is the square of the side
length as measured in wavelengths,
so that the net algorithmic cost is $O(\Nt (M + N^2 \log N))$.
The convergence order is controlled by that of the 1D temporal interpolations, which
may be made very high. All other aspects of convergence are spectral,
since they are controlled by Fourier truncation and quasi-optimal blending
functions (smooth partitions of unity).
The use of Fourier representations avoids numerical grid-based dispersion errors.
Large-scale examples with around one million sources and targets
covering 90,000 square wavelengths
were shown to be computable on a
single compute node
with six digits of precision.

This more elaborate 2D work complements recent WFP evaluation algorithms in
1D \cite{wfp2025} and 3D \cite{tkwfp3d}.
We believe that the 2D and 3D TK-WFP evaluation algorithms can 
serve as the key ingredient in the
efficient time marching of potential-theoretic solutions of challenging time-domain wave scattering problems.

\bibliographystyle{siamplain}
\bibliography{refs}
 
\end{document}